\documentclass[a4paper]{amsart}
\usepackage{amssymb,amsthm,mathtools,enumerate}
\usepackage{multirow}
\usepackage{tikz,tikzscale}
\usepackage{caption,subcaption}

\usepackage{longtable}
\usepackage{placeins}

\newtheorem{theorem}{Theorem}
\newtheorem{lemma}{Lemma}[section]
\newtheorem{corollary}{Corollary}

\theoremstyle{definition}
\newtheorem{definition}{Definition}
\newtheorem*{remark}{Remark}

\usetikzlibrary{calc}

\definecolor{diagrambg}{gray}{1}
\definecolor{diagramfg}{gray}{0}

\pagecolor{diagrambg}
\color{diagramfg}
\colorlet{knotbg}{diagrambg}
\colorlet{knotfg}{diagramfg}

\tikzset{
  local binding/.style={
    draw=red,
    dashed,
    rounded corners,
    line width=1.5pt
  },
  local strand/.style={
    draw=diagrambg,
    double=diagramfg,
    line width=0.8pt,
    double distance=1pt,
    rounded corners=6pt
  },
  label/.style={
    text=diagramfg
  }
}

\newcommand{\Figure}[1]{Figure \ref{fig:#1}}
\newcommand{\Fig}[1]{Fig. \ref{fig:#1}}
\newcommand{\Table}[1]{Table \ref{tab:#1}}
\newcommand{\Lemma}[1]{Lemma \ref{L:#1}}
\newcommand{\Theorem}[1]{Theorem \ref{T:#1}}
\newcommand{\Definition}[1]{Definition \ref{D:#1}}
\newcommand{\Section}[1]{\S\ref{S:#1}}
\makeatletter
\newcommand{\setFloat}[1]{%
  \renewcommand{\fps@figure}{#1}%
}
\makeatother
\newcommand{\nextLabel}[1]{%
  \providecommand{\pendingLabel}{}%
  \renewcommand{\pendingLabel}{#1}%
}

\newcommand{\ttt}[1]{\ensuremath{\texttt{#1}}}
\newcommand{\ottt}[1]{\ensuremath{\overline{\texttt{#1}}}}

\title{Locally Minimal Bridge Presentations of Knots}
\subjclass[2020]{Primary 57K10}
\author{Chad Musick}
\email{chad.e.musick@gmail.com}

\begin{document}
\begin{abstract}
  For classical knots, many characteristics are difficult to determine from arbitrary diagrams
  of the knot. Bridge number is of this type. It is possible to have a
  bridge presentation of a knot in which the number of bridges is only locally minimal, as shown in
  \cite{OzawaTakaoLocalBridge}. Here, we give a method for beginning with an arbitrary classical
  knot diagram having \(n\) crossings and producing a locally minimal bridge presentation of the
  knot with size \(O(n^2)\). This method requires only polynomial time and space.
  For some classes of knots, any local minimum is also the global minimum, which was shown for
  the trivial knot in \cite{Otal1982}. Because the unknot is the only knot with bridge
  number \(1\), unknot recognition is in \(P\).
\end{abstract}
\maketitle

A problem of some interest in knot theory is to find knot diagrams with a minimal
number of bridges. Ozawa and Takao \cite{OzawaTakaoLocalBridge} proved that diagrams can
have a locally but not globally minimal number of bridges for some knots. Despite this,
some classes of knots are known to have no bridge presentations that are only locally
minimal. Otal's earlier result \cite{Otal1982} establishes that no diagram of the trivial
knot has a locally minimal bridge count higher than \(1\).

In writing \cite{MinimalBridges}, this author used the method described here to find an upper
bound for every \(11\)-crossing prime knot and a diagram realizing that bound. In combination
with lower bounds known from other knot characteristics, this completed the catalog of the bridge
numbers of all \(11\)-crossing prime knots. Here, we present the method in a way that keeps the
time and space complexity bounded polynomially relative to the number of crossings in the original
diagram, although we do not prove a non-exponential bound on the number of crossings of the
diagram realizing this locally minimal bridge position.

In \cite{TarLinks}, this author described a grammar to characterize embeddings of links in \(S^3\).
We impose some restrictions on that grammar here in order to make topological possibilities
legible directly from that characterization, without geometric reconstruction. We briefly
show a variant of the earlier encoding, specialized for the purpose of characterizing
bridges and underpass-avoidance moves in knot diagrams.

In this paper, we establish a combinatorial grammar that can be used to begin with an
arbitrary knot diagram, characterize it, and find a locally minimal
bridge presentation in polynomial time and space relative to the number of crossings in
the initial diagram.

\section{Bridge Presentations and \(3\)-Page Bridge Embeddings}
\begin{definition}[bridge presentation]\label{D:bridgepresentation}
We define a \emph{bridge presentation} in the following way, which constructs a Heegaard splitting
between two \(3\)-balls.

Identify \(S^3\) with \(\mathbb{R}^3\cup\{\infty\}\). Let
\begin{align*}
B^- &= \left\{x \in \mathbb{R}^3: \lVert x\rVert \leq \frac{3}{2} \right\}\text{;} &
B^+ &= \left\{x \in \mathbb{R}^3: \lVert x\rVert \geq \frac{3}{2} \right\} \cup \{\infty\}\text{.}
\end{align*}
Then
\[ \Sigma = B^-\cap B^+\text{;}\qquad S^3 = B^- \cup_{\Sigma} B^+ \]
is the standard genus-zero Heegaard splitting of \(S^3\), with Heegaard sphere \(\Sigma\).

This is a bridge presentation when each of \(B^-\) and \(B^+\) contains only a trivial tangle.
\end{definition}

A combinatorial version of this is a \(3\)-page embedding that keeps the underpasses on their
own page and bridges on a pair of pages. See Dynnikov \cite{DynnikovThreePage, DynnikovUniversalSemigroup}
for a general definition of a \(3\)-page embedding of a link and some algebraic results
for such embeddings. Here, we perform the one-point compactification of the binding line by
adding a point at \(\infty\), and we restrict the arc placement to keep the
embedding in bridge-presentation form.

\begin{definition}[\(3\)-page bridge embedding]
  A \emph{\(3\)-page bridge embedding} of a knot \(K\) is an embedding of \(K\) into \(3\) pages,
  which we label \(U\), \(S\), and \(N\) for \emph{underpasses}, \emph{south}, and
  \emph{north}, respectively, with the following characteristics.
  \begin{itemize}
    \item The binding line with a point added at \(\infty\) is the \emph{binding equator}.
    \item Each page is a disc.
    \item No page has any crossings between arcs.
    \item The binding equator of the pages has a cyclic order, along the edge of each page.
    \item Joining pages \(S\) and \(N\) forms a copy of \(S^2\) with the binding equator along the
      equator of that sphere.
    \item Each arc in \(U\) lies along the equator. These arcs may skip punctures but are
      pairwise disjoint.
    \item The interior of each arc in \(S \cup N\) is entirely within one of the two pages.
    \item Every puncture on the binding has exactly zero or two incident edges and at most one
      incident edge in any one page.
    \item The curve traced by following the arcs through the pages is a single loop.
  \end{itemize}
\end{definition}

We note that the \(3\)-page bridge embedding is equivalent to a bridge presentation by
putting page \(U\) on one hemisphere of a sphere with radius \(1\) and center
at the origin of \(\mathbb{R}^3\), putting pages \(S\) and \(N\) as the two hemispheres of a
sphere with radius \(2\) and center at the origin of \(\mathbb{R}^3\), aligning the
equators of the two spheres to lie on the plane \(z = 0\), and drawing
segments between punctures incident to arcs at different radii along rays from the origin.
This matches \Definition{bridgepresentation}.

\begin{lemma}
  Given a knot \(K\) in a \(3\)-page bridge embedding, \(K\) is a bridge presentation.
\end{lemma}
\begin{proof}
  \(\Sigma\) is a bridge sphere of \(K\) with \(2n\) punctures. By construction, each
  component of \(K\cap B^-\) and \(K\cap B^+\) is a trivial tangle.
\end{proof}

In the proofs that follow, vertices that are incident to an edge in page \(U\) are treated
differently than those that are not, so we name them.

\begin{definition}\label{D:terminal vertex}
  A \emph{terminal vertex} is a vertex with an incident edge in page \(U\). A \emph{non-terminal
  vertex} is a vertex with no incident edge in page \(U\).
\end{definition}

We will later use an embedding with weighted edges for notational and computational convenience,
but that notation is equivalent to this \(3\)-page bridge embedding and so to a bridge presentation
of the knot.

\begin{lemma}\label{L:quotient}
  Given binding-adjacent non-terminal vertices \(v_1\) and \(v_2\) of a \(3\)-page bridge
  embedding with no arc \(v_1v_2\), quotienting together \(v_1\) and \(v_2\) is reversible.
\end{lemma}
\begin{proof}
  See \Fig{quotient}.
  If either \(v_1\) or \(v_2\) is empty, this is trivially true, so suppose neither is empty.

  Without loss of generality, let one path include \(av_1b\) and the other include \(cv_2d\),
  and assume \(a < b < c < d\) in some linearization of the cyclic binding order.
  Because these do not cross, \(v_1 \leq v_2\). Recovering the pair \(avb, cvd\) as \(av_2b, cv_1d\)
  would force \(v_2 < v_1\), but \(v_1 \leq v_2\), so the reconstruction is
  well-defined.
\end{proof}

\setFloat{htbp}
\nextLabel{fig:quotient}
\begin{figure}
  \centering
  \begin{subfigure}[t]{0.31\textwidth}
    \centering
    \begin{tikzpicture}
      \node (a) at (0, 2) {\(a\)};
      \node (b) at (0, 0) {\(b\)};
      \node (c) at (2, 2) {\(c\)};
      \node (d) at (2, 0) {\(d\)};
      \node (v1) at (0.33, 1) {\(v_1\)};
      \node (v2) at (1.67, 1) {\(v_2\)};
      \draw (a) -- (v1) -- (b);
      \draw (c) -- (v2) -- (d);
      \draw [local binding] (0, 1) -- (v1) -- (v2) -- (2, 1);
      \draw [local binding] (a) -- (c);
      \draw [local binding] (b) -- (d);
    \end{tikzpicture}
    \caption{Before vertex quotient}
  \end{subfigure}
  \hfill
  \begin{subfigure}[t]{0.31\textwidth}
    \centering
    \begin{tikzpicture}
      \node (a) at (0, 2) {\(a\)};
      \node (b) at (0, 0) {\(b\)};
      \node (c) at (2, 2) {\(c\)};
      \node (d) at (2, 0) {\(d\)};
      \node (v) at (1, 1) {\(v\)};
      \draw (a) -- (v) -- (b);
      \draw (c) -- (v) -- (d);
      \draw [local binding] (0, 1) -- (v) -- (2, 1);
      \draw [local binding] (a) -- (c);
      \draw [local binding] (b) -- (d);
    \end{tikzpicture}
    \caption{With vertex quotient}
  \end{subfigure}
  \hfill
  \begin{subfigure}[t]{0.31\textwidth}
    \centering
    \begin{tikzpicture}
      \node (a) at (0, 2) {\(a\)};
      \node (b) at (0, 0) {\(b\)};
      \node (c) at (2, 2) {\(c\)};
      \node (d) at (2, 0) {\(d\)};
      \node (v1) at (0.33, 1) {\(v_1\)};
      \node (v2) at (1.67, 1) {\(v_2\)};
      \draw (a) -- (v2) -- (b);
      \draw (c) -- (v1) -- (d);
      \draw [local binding] (0, 1) -- (v1) -- (v2) -- (2, 1);
      \draw [local binding] (a) -- (c);
      \draw [local binding] (b) -- (d);
    \end{tikzpicture}
    \caption{Erroneous reconstruction violates non-crossing}
  \end{subfigure}
  \caption{Visual proof of \Lemma{quotient}. Vertices \(v_1\) and \(v_2\) are
    adjacent by assumption and non-terminal, so they have an even number of incident
    edges. The other vertices may be terminal or non-terminal.}\label{\pendingLabel}
\end{figure}
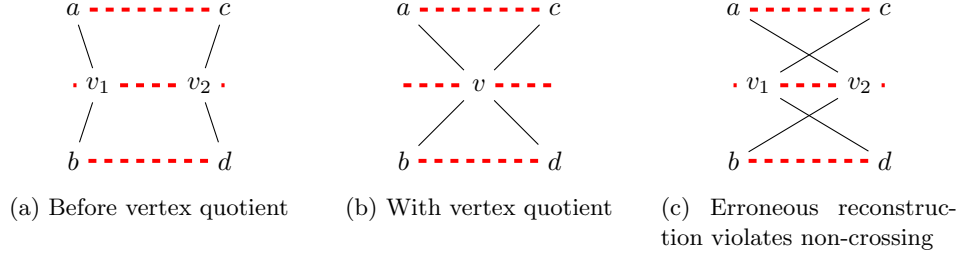

Later, we will drop the requirement that the adjacent non-terminal vertices have no edge
between them (see \Lemma{zipper}) and maintain the same knot, but at the cost of not
maintaining the same embedding.

\subsection{Construction of the \(3\)-Page Bridge Embedding}\label{S:construction}
For a knot shadow with \(n\) crossings, the existence of a subhamiltonian cycle of the graph
vertices and an \(O(n^2)\)-time algorithm to find that cycle are proven by Bekos,
Gronemann, and Raftopoulou \cite[Theorem 3]{Hamiltonian}. We use this result to construct a \(3\)-page bridge
embedding of the knot diagram, which is particularly nice to work with for bridge reduction.

We begin with a standard knot diagram \(K_D\). Using the method of \cite{Hamiltonian},
we construct a subhamiltonian cycle on the intersections of the knot shadow. These intersections
correspond to the knot crossings, so we have constructed a subhamiltonian cycle on the crossings.
This cycle is almost the binding equator. Creating the binding equator requires only separating
the pieces onto the three pages of our embedding. We do this now.

\Figure{crossing} shows the different cases for the modification to the binding equator.
In all cases, general position gives us either \(3\) or \(4\) intersections between the binding equator and
the knot. We use \(2\) as bridge ends for each crossing, \(1\) point for the crossing of the
bridge over the binding, and in some cases \(1\) additional point to exit via the desired face.
In all cases, the bridge alternates between sides of the binding equator, corresponding with pages
\(S\) and \(N\) and maintaining the necessary matching condition between the pages.

In words: suppose without loss of generality that the crossing runs up--down and left--right,
dividing the local neighborhood into \(4\) quadrants, which we give the conventional names
counter-clockwise from the top right corner. Let the left--right strand be the
underpassing strand. Place cuts left and right of the crossing. Suppose that the binding equator
arrives through quadrant I. In all cases, it should touch both cuts in sequence to allow pushing
the underpass to the boundary of page \(U\), then depart through one of the quadrants. Every
quadrant has one cut mark, so we lose no generality by assuming one end of the binding equator
is in quadrant I. We examine cases for the four possible departure quadrants. If we depart through
quadrant I, touch right, left, then cross up, and depart. If we depart through quadrant II
or III, touch right, left, and depart. If we depart through quadrant IV, touch right, left,
then cross down, and depart.

\nextLabel{fig:crossing}
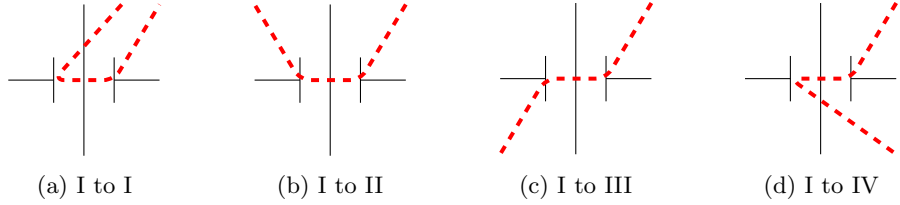
\begin{figure}
  \centering
  \begin{subfigure}[t]{0.23\textwidth}
    \centering
    \begin{tikzpicture}
      \coordinate (l) at (0, 1);
      \coordinate (r) at (2, 1);
      \coordinate (u) at (1, 2);
      \coordinate (d) at (1, 0);
      \draw (l) -- (0.6, 1);
      \draw (1.4, 1) -- (r);
      \draw (0.6, 1.3) -- (0.6, 0.7);
      \draw (1.4, 1.3) -- (1.4, 0.7);
      \draw (u) -- (d);
      \draw[local binding] (1.5, 2) -- (0.6, 1) -- (1.4, 1) -- (2, 2);
    \end{tikzpicture}
    \caption{I to I}
  \end{subfigure}
  \hfill
  \begin{subfigure}[t]{0.23\textwidth}
    \centering
    \begin{tikzpicture}
      \coordinate (l) at (0, 1);
      \coordinate (r) at (2, 1);
      \coordinate (u) at (1, 2);
      \coordinate (d) at (1, 0);
      \draw (l) -- (0.6, 1);
      \draw (1.4, 1) -- (r);
      \draw (0.6, 1.3) -- (0.6, 0.7);
      \draw (1.4, 1.3) -- (1.4, 0.7);
      \draw (u) -- (d);
      \draw[local binding] (2, 2) -- (1.4, 1) -- (0.6, 1) -- (0, 2);
    \end{tikzpicture}
    \caption{I to II}
  \end{subfigure}
  \hfill
  \begin{subfigure}[t]{0.23\textwidth}
    \centering
    \begin{tikzpicture}
      \coordinate (l) at (0, 1);
      \coordinate (r) at (2, 1);
      \coordinate (u) at (1, 2);
      \coordinate (d) at (1, 0);
      \draw (l) -- (0.6, 1);
      \draw (1.4, 1) -- (r);
      \draw (0.6, 1.3) -- (0.6, 0.7);
      \draw (1.4, 1.3) -- (1.4, 0.7);
      \draw (u) -- (d);
      \draw[local binding] (2, 2) -- (1.4, 1) -- (0.6, 1) -- (0, 0);
    \end{tikzpicture}
    \caption{I to III}
  \end{subfigure}
  \hfill
  \begin{subfigure}[t]{0.23\textwidth}
    \centering
    \begin{tikzpicture}
      \coordinate (l) at (0, 1);
      \coordinate (r) at (2, 1);
      \coordinate (u) at (1, 2);
      \coordinate (d) at (1, 0);
      \draw (l) -- (0.6, 1);
      \draw (1.4, 1) -- (r);
      \draw (0.6, 1.3) -- (0.6, 0.7);
      \draw (1.4, 1.3) -- (1.4, 0.7);
      \draw (u) -- (d);
      \draw[local binding] (2, 2) -- (1.4, 1) -- (0.6, 1) -- (2, 0);
    \end{tikzpicture}
    \caption{I to IV}
  \end{subfigure}

  \caption{Construction of travel for the binding equator in a \(3\)-page bridge embedding,
  with the horizontal gap representing the underpass.}\label{\pendingLabel}
\end{figure}

\begin{lemma}\label{L:firstLength}
  The construction using the subhamiltonian cycle produces a \(3\)-page bridge embedding, with
  each crossing contributing \(1\) or \(2\) crossings between a bridge and the binding equator,
  giving at most \(4n\) letters to describe all bridges of an \(n\)-crossing knot diagram.
\end{lemma}
\begin{proof}
  By construction, every underpass lies on the boundary of the binding equator and can be pushed
  to page \(U\) with the anti-matching condition. Also by construction, every bridge starts on the
  binding equator, alternates its arcs between pages \(S\) and \(N\), then ends on the binding
  equator with the anti-matching condition.

  The count of length contributions from bridges is a direct consequence of the sum of \(n\)
  crossings, each adding no more than \(2\) arcs for bridges plus at most \(2\) letters for the
  bridge ends.
\end{proof}

Having this embedding and, from \Lemma{quotient}, that we do not need to explicitly
distinguish between adjacent non-terminal vertices, we can give a simple
grammar that characterizes the diagram, the embedding, and the bridge presentation simultaneously.

In \cite{TarLinks}, this author described a grammar to characterize embeddings of links in \(S^3\).
There are three relevant definitions for the present problem of bridge-count minimization.

\begin{definition}[embedding grammar]
  Each bridge endpoint (equivalently, underpass endpoint) is given a distinct label, its
  \emph{letter}.

  The binding equator is written as the \emph{circle word}. This word begins with a semicolon
  symbol \ttt{;} and then lists, in order as they occur moving eastward (anti-clockwise) on the
  binding equator, the names of the terminal vertices.

  The path of each bridge is described by the \emph{bridge word} of that bridge. The bridge word
  begins with either \ttt{+} or \ttt{-}, indicating that its first arc lies in page \(N\)
  or page \(S\), respectively. Each point at which the interior of a bridge crosses the binding
  equator is listed as the symbol of the terminal vertex nearest in the westward direction.

  By construction, the path of each underpass is exactly one arc between adjacent terminal
  vertices, and we omit these because they can be reconstructed from the circle word.
\end{definition}

An example is helpful. \Figure{trefoil} shows a trefoil in a \(3\)-page bridge embedding.
In the figure, the underpasses are omitted but can be read from the circle word. They lie on the
boundary of page \(U\), joining \ttt{0} with \ttt{1}, \ttt{2} with \ttt{3}, and
\ttt{4} with \ttt{5}. The interior of the circle shows page \(N\), and the exterior shows
page \(S\). Moving counter-clockwise around the circle is moving east on the binding equator.

\nextLabel{fig:trefoil}
\begin{figure}[!htbp]
  \centering
  \begin{tikzpicture}
    \draw[local binding] (0, 0) circle[radius=1];
    \node[fill=diagrambg, inner sep=2pt] (v1) at (0:1) {0};
    \node[fill=diagrambg, inner sep=2pt] (v2) at (80:1) {1};
    \node[fill=diagrambg, inner sep=2pt] (v5) at (120:1) {4};
    \node[fill=diagrambg, inner sep=2pt] (v6) at (200:1) {5};
    \node[fill=diagrambg, inner sep=2pt] (v3) at (240:1) {2};
    \node[fill=diagrambg, inner sep=2pt] (v4) at (320:1) {3};
    \draw[local strand] (v2) -- (80:0.5) arc[start angle=80, end angle=160, radius=0.5] -- (160:1.5) arc[start angle=160, end angle=240, radius=1.5] -- (v3);
    \draw[local strand] (v4) -- (320:0.5) arc[start angle=320, end angle=400, radius=0.5] -- (40:1.5) arc[start angle=40, end angle=120, radius=1.5] -- (v5);
    \draw[local strand] (v6) -- (200:0.5) arc[start angle=200, end angle=280, radius=0.5] -- (280:1.5) arc[start angle=280, end angle=360, radius=1.5] -- (v1);
  \end{tikzpicture}
  \caption{Trefoil in a \(3\)-page bridge embedding with sentence \ttt{+142+304+520;014523}.
    The sentence \ttt{+142+304+520;145230} encodes a trivial knot by forcing a different
    pairing of vertices to underpasses.}
  \label{\pendingLabel}
\end{figure}
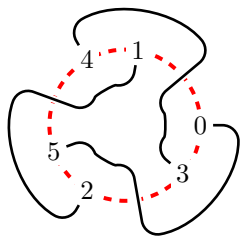

In this paper, we impose some restrictions on the grammar not present in the fully general
grammar of \cite{TarLinks}. First, we consider only ordinary crossings, omitting the other types
of crossings (e.g., virtual crossings). Second, we require that the \emph{word}
of each underpass \(ab\) is \ttt{+ab} with \(a\) and \(b\) adjacent terminal vertices
on the binding equator (ignoring intervening non-terminal vertices). This makes the words'
topological possibilities legible directly from that characterization, without geometric
reconstruction, and we omit the underpass words. Finally, we require that the circle word's
first two letters are the ends of a shared underpass, which prevents ambiguity about matching
for the ends of underpasses.

\section{Metagrammar Rules}
We describe a metagrammar here, which can be turned into a full grammar according to the specific
diagram, and this metagrammar is unambiguous due to \Lemma{quotient}. See \Table{symbols}
for the meaning of symbols when not otherwise defined.

\nextLabel{tab:symbols}
\begin{table}
  \caption{Metagrammar symbols}\label{\pendingLabel}
  \begin{tabular}{l|l}
    \textbf{Symbol} & \textbf{Meaning}\\
    \hline
    \ottt{a} & Terminal vertex name\\
    \ttt{a} & Non-terminal vertex (adjacent to \ottt{a} if relevant)\\
    \(\ttt{a}^{\circ}, \ttt{a}'\) & Vertex \ttt{a} on (resp., not on) underpass with end \ottt{a}\\
    \(\overrightarrow{\ttt{W}}\) & Well-formed fragment of a word\\
    \ttt{A} & Bridge word starting or ending with vertex \ottt{a}\\
    \ottt{ab} & Underpass with ends \ottt{a} and \ottt{b}\\
    \(\pm\) & Sign, as \ttt{+} or \ttt{-}\\
    \(\mp\) & Opposite sign of \(\pm\)\\
    \(\ttt{A}'\) & Bridge-looping word of bridge \ttt{A} (see \Definition{bridge-loop})\\
    \ttt{A\#B} & Concatenated value of \ttt{A} and \ttt{B} (see \Definition{concatenation})\\
    \(\emptyset\) & \text{Empty string}
  \end{tabular}
\end{table}

We have four types of metagrammar rules, listed in \Table{rules}. The normalization,
quotient, and underpass closure rules are eagerly applied. The underpass closure move is
conditioned on whether either of a specific pair of symbols appears in any bridge word.
Underpass avoidance is likewise conditioned on the absence of a specific symbol
in at least one of the bridges leading from the underpass. No rule introduces new symbols or
revives previously eliminated symbols, which will be important to containing the complexity
of the eventual construction. We note that none of the rules requires a specific symbol count,
though some require knowing the presence (or absence) of a symbol in a word or words and some
require knowing the odd/even parity of the number of occurrences of a symbol.

We note here that the choice of non-eager moves and, specifically, which order to perform underpass
avoidances and via which bridge, results in a system that is \emph{not} proven to be confluent.
That is, different choices and different ordering may result in different end states,
which is why we can guarantee only local minimality of the number of bridges in the final diagram
in light of the results from Ozawa and Takao \cite{OzawaTakaoLocalBridge} showing that local
bridge minimality without global bridge minimality is possible.

\nextLabel{tab:rules}
{\newcommand{\blankrow}{\rule{0pt}{0.5ex} & & & \\[-1ex]}
\begin{table}
  \caption{Metagrammar rules}\label{\pendingLabel}
  \begin{tabular}{l|rl|l}
    \textbf{Rule name} & \multicolumn{2}{c|}{\textbf{Rule}} & \textbf{Notes}\\
    \hline
    \blankrow
    \multirow{3}{*}{Normalization} & \(\pm\ottt{a}\ttt{a}\) & \(\mapsto \mp\ottt{a}\) & Removes trivial initial arc\\
    & \ttt{a}\ottt{a} & \(\mapsto \ottt{a}\) & Removes trivial terminal arc\\
    & \ttt{aa} & \(\mapsto \emptyset\) & Removes trivial internal arc\\
    \blankrow
    \hline
    \blankrow
    \multirow{4}{*}{Loop snipping} & \(\ttt{a}\overrightarrow{\ttt{W}}\ttt{b}\) & \(\mapsto \ttt{ab}\) &
      \multirow{4}{*}{\shortstack[l]{Requires: Forms an empty loop;\\\(\overrightarrow{\ttt{W}}\) has positive, even length;\\(see \Definition{loop-snipping}).}}\\
    & \(\ottt{a}\overrightarrow{\ttt{W}}\ttt{b}\) & \(\mapsto \ottt{a}\ttt{b}\) &\\
    & \(\ttt{b}\overrightarrow{\ttt{W}}\ottt{a}\) & \(\mapsto \ttt{b}\ottt{a}\) &\\
    & \(\ottt{a}\overrightarrow{\ttt{W}}\ottt{b}\) & \(\mapsto \ottt{a}\ottt{b}\) &\\
    \blankrow
    \hline
    \blankrow
    Underpass avoidance & \(\ttt{a}^{\circ}\) & \(\mapsto \ttt{A}'\) &
      Requires: \(\ttt{a}^{\circ} \notin \ttt{A}\)\\
    \blankrow
    \hline
    \blankrow
    Underpass closure & \(\ttt{A}, \ttt{B}\) & \(\mapsto \ttt{A\#B}\)
      & Requires: \(\exists\,\ottt{ab}; \forall\ttt{C}\ \,\ttt{a}^{\circ},\ttt{b}^{\circ}\notin \ttt{C}\)\\
    \blankrow
    \hline
    \blankrow
    Quotient & \ttt{a} & \(\mapsto \ttt{b}\) & Requires: \ttt{a} and \ttt{b} adjacent
  \end{tabular}
\end{table}}

\begin{definition}[bridge concatenation]\label{D:concatenation}
  Let \ttt{A} and \ttt{B} be bridge words sharing an underpass \ottt{ab},
  with \(\ttt{A} = \pm\overrightarrow{\ttt{V}}\ottt{a}\) for a well-formed fragment \(\overrightarrow{\ttt{V}}\)
  and \(\ttt{B} = \pm\ottt{b}\overrightarrow{\ttt{W}}\) for a well-formed fragment \(\overrightarrow{\ttt{W}}\).
  The bridge concatenation \ttt{A\#B} is \(\pm\overrightarrow{\ttt{V}}\overrightarrow{\ttt{W}}\) (resp., \(\pm\overrightarrow{\ttt{V}}\ttt{a}\overrightarrow{\ttt{W}}\))
  if the final arc of \ttt{A} and the initial arc of \ttt{B} share (resp., do not share)
  a page.
\end{definition}
\begin{remark}
  We may, if needed, replace any bridge word starting at \ottt{a} and ending at \ottt{b}
  with a word starting at \ottt{b} and ending at \ottt{a} by reversing the order of the
  letters and, if the number of letters is odd, swapping the sign. These two encode
  the same arcs on the same pages.
\end{remark}

\begin{definition}[bridge-looping word]\label{D:bridge-loop}
  The \emph{bridge-looping word} \(\ttt{A}'\) is the characterization of a
  trivial knot passing around bridge \ttt{A} and around no other bridge.
\end{definition}
\begin{remark}
  Let \(\pm\ottt{a}\ttt{W}\ottt{c}\) be a bridge word \ttt{A} or
  \ttt{C} with well-formed (possibly empty) fragment \ttt{W}. Because a bridge can
  be named from either end, \ttt{A} and \ttt{C} are the same word if and only if
  a bridge runs between \(\ottt{a}\) and \(\ottt{c}\).
  Let \(\overleftarrow{\ttt{W}}\) be \ttt{W} in reversed order. A characterization
  of \(\ttt{A}'\) is either
  \(\pm\ttt{a}^{\circ}\ttt{Wc}^{\circ}\ttt{c}'\overleftarrow{\ttt{W}}\ttt{a}'\ttt{a}^{\circ}\) or
  \(\pm\ttt{a}^{\circ}\ttt{Wc}'\ttt{c}^{\circ}\overleftarrow{\ttt{W}}\ttt{a}'\ttt{a}^{\circ}\).
  Which is correct depends on the odd/even parity of the bridge-word length and the relative
  positions of the two vertices bounding the underpass incident to \ttt{c}. One of these
  words will cross \ttt{A} and the other will not.

  Because \(\ttt{A}'\) is the characterization of a trivial knot, it forms a \(1\)-bridge
  cycle and lives on pages \(S \cup N\). Thus, it has no terminal vertices.
  For use in substitution, the bridge-looping word will be in one of two positions:
  (1) The position written here; (2) The opposite orientation, found by reversing the order of
  the letters and swapping the sign.

  The necessity to consider various cases for the bridge-looping word is a consequence of choosing
  a linearization and orientation of an unoriented cycle. We note that when underpass
  \ottt{ab} exists, the vertices \(\ttt{a}^{\circ}\) and \(\ttt{b}^{\circ}\)
  are the same vertex. The difference between using \(\ttt{A}'\) and using \(\ttt{B}'\)
  is which of the two bridges incident to \ottt{ab}
  is used to avoid \ottt{ab}.

  If a bridge-looping word is considered as a set of arcs each characterized
  by a pair of vertices and a page, then the two possible words \emph{as geometry} are identical.
  The difference is because the representation as arc counts does not encode an orientation
  or a start.

  We explicitly note that the length of the replacement word for a bridge with \(k\) letters
  is \(2k + 1\). This matters for the length of the grammar produced by the metagrammar.

  When used for underpass avoidance as a replacement, the initial and terminal \(\ttt{a}^{\circ}\) are 
  omitted because we are splitting letters \(\ttt{xa}^{\circ}\ttt{y}\) into arcs \ttt{xa} and
  \ttt{ay} then splicing in the new word, so normalization would immediately eliminate these
  letters. The growth is therefore by \(2k - 1\) letters before any other normalization that may
  occur. See \Definition{bridge-loop-arcs} for the definition in arc form, which is more
  natural and does not have special cases.
\end{remark}

\Figure{avoid} shows the geometry corresponding to the grammatical rule and why this
is useful. By putting the additional trivial knot into the diagram looping \ttt{B} and
performing standard knot composition, the underpass incident to \ottt{b} is
avoided. Because this is knot composition with a trivial knot, the type of knot is unchanged.
Alternatively, but with the same result, we can view the avoidance as an ambient isotopy of the
arc crossing over the underpass or, to use Tawn's characterization in
\cite[\S4]{TawnCutComplex,TawnErratum}, as the replacement of an arc by forming a loop using the arc and
substituting the complement in that loop for the original arc.

No matter how it is viewed, the knot type and the bridge presentation are preserved, up to
the unambiguous recovery of vertices by \Lemma{quotient}, but the diagram is obviously
changed.

\nextLabel{fig:avoid}
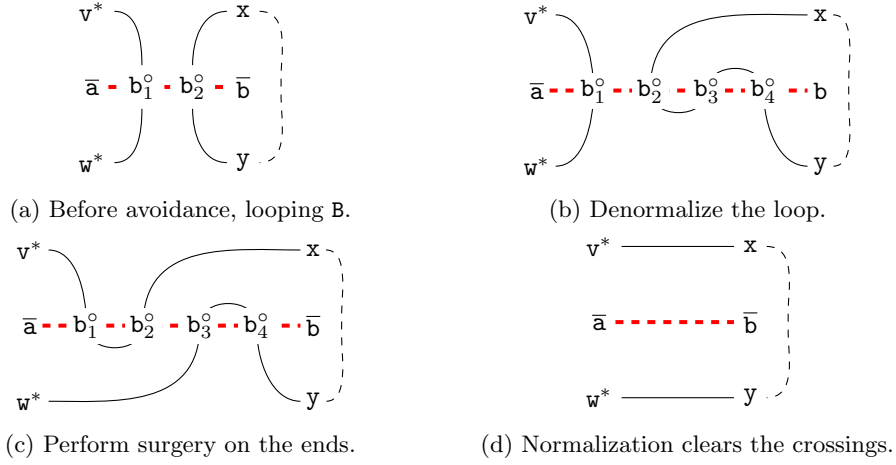
\begin{figure}
  \centering
  \begin{subfigure}[t]{0.47\textwidth}
    \centering
    \begin{tikzpicture}
      \node (v) at (0, 2) {\(\ttt{v}^*\)};
      \node (w) at (0, 0) {\(\ttt{w}^*\)};
      \node (x) at (2, 2) {\ttt{x}};
      \node (y) at (2, 0) {\ttt{y}};
      \node (a) at (0, 1) {\ottt{a}};
      \node (b) at (2, 1) {\ottt{b}};
      \coordinate (c) at (2.5, 1);
      \node (b1) at (0.66, 1) {\(\ttt{b}_1^{\circ}\)};
      \node (b2) at (1.33, 1) {\(\ttt{b}_2^{\circ}\)};
      \draw (v) to[out=0, in=90] (b1) to[out=270, in=0] (w);
      \draw (x) to[out=180, in=90] (b2) to[out=270, in=180] (y);
      \draw [local binding] (a) -- (b1) -- (b2) -- (b);
      \draw[dashed, rounded corners] (x) to[out=0, in=90] (c) to[out=270, in=0] (y);
    \end{tikzpicture}
    \caption{Before avoidance, looping \ttt{B}.}
  \end{subfigure}
  \hfill
    \begin{subfigure}[t]{0.47\textwidth}
    \centering
    \begin{tikzpicture}[xscale=0.75, every node/.append style={fill=diagrambg, inner sep=2pt}]
        \node (v) at (0, 2) {\(\ttt{v}^*\)};
        \node (w) at (0, 0) {\(\ttt{w}^*\)};
        \node (x) at (5, 2) {\ttt{x}};
        \node (y) at (5, 0) {\ttt{y}};
        \node (a) at (0, 1) {\ottt{a}};
        \node (b) at (5, 1) {\ttt{b}};
        \coordinate (c) at (5.5, 1);
        \coordinate (b1) at (1, 1);
        \coordinate (b2) at (2, 1);
        \coordinate (b3) at (3, 1);
        \coordinate (b4) at (4, 1);
        \draw (v) to[out=0, in=90] (b1) to[out=270, in=0] (w);
        \draw (x) to[out=180, in=90] (b2) to[out=270, in=270] (b3)
        to [out=90, in=90] (b4) to[out=270, in=180] (y);
        \draw[local binding] (a) -- (b1) -- (b);
        \node at (b1) {\(\ttt{b}^{\circ}_1\)};
        \node at (b2) {\(\ttt{b}^{\circ}_2\)};
        \node at (b3) {\(\ttt{b}^{\circ}_3\)};
        \node at (b4) {\(\ttt{b}^{\circ}_4\)};
        \draw[dashed, rounded corners] (x) to[out=0, in=90] (c) to[out=270, in=0] (y);
    \end{tikzpicture}
    \caption{Denormalize the loop.}
    \end{subfigure}
    \begin{subfigure}[t]{0.47\textwidth}
    \centering
    \begin{tikzpicture}[xscale=0.75, every node/.append style={fill=diagrambg, inner sep=2pt}]
        \node (v) at (0, 2) {\(\ttt{v}^*\)};
        \node (w) at (0, 0) {\(\ttt{w}^*\)};
        \node (x) at (5, 2) {\ttt{x}};
        \node (y) at (5, 0) {\ttt{y}};
        \node (a) at (0, 1) {\ottt{a}};
        \node (b) at (5, 1) {\ottt{b}};
        \coordinate (c) at (5.5, 1);
        \coordinate (b1) at (1, 1);
        \coordinate (b2) at (2, 1);
        \coordinate (b3) at (3, 1);
        \coordinate (b4) at (4, 1);
        \draw (w) to[out=0, in=270] (b3);
        \draw (v) to[out=0, in=90] (b1);
        \draw (b2) to[out=270, in=270] (b1);
        \draw (x) to[out=180, in=90] (b2);
        \draw (b3) to [out=90, in=90] (b4) to[out=270, in=180] (y);
        \draw[local binding] (a) -- (b1) -- (b);
        \node at (b1) {\(\ttt{b}^{\circ}_1\)};
        \node at (b2) {\(\ttt{b}^{\circ}_2\)};
        \node at (b3) {\(\ttt{b}^{\circ}_3\)};
        \node at (b4) {\(\ttt{b}^{\circ}_4\)};
        \draw[dashed, rounded corners] (x) to[out=0, in=90] (c) to[out=270, in=0] (y);
    \end{tikzpicture}
    \caption{Perform surgery on the ends.}
    \end{subfigure}
  \hfill
  \begin{subfigure}[t]{0.47\textwidth}
    \centering
    \begin{tikzpicture}
      \node (v) at (0, 2) {\(\ttt{v}^*\)};
      \node (w) at (0, 0) {\(\ttt{w}^*\)};
      \node (x) at (2, 2) {\ttt{x}};
      \node (y) at (2, 0) {\ttt{y}};
      \node (a) at (0, 1) {\ottt{a}};
      \node (b) at (2, 1) {\ottt{b}};
      \coordinate (c) at (2.5, 1);
      \draw (v) -- (x);
      \draw (w) -- (y);
      \draw [local binding] (a) -- (b);
      \draw[dashed, rounded corners] (x) to[out=0, in=90] (c) to[out=270, in=0] (y);
    \end{tikzpicture}
    \caption{Normalization clears the crossings.}
  \end{subfigure}
  \caption{Geometry of avoiding underpass \(\overline{\ttt{ab}}\) by composition with
    \(\|\ttt{v}^*\|\) copies of the
    loop around bridge \ttt{B}. The loop is a trivial knot, and the dashed
    arc indicates the remainder of that loop. By construction, vertices
    \ottt{a} and \ottt{b} are terminal, \(\ttt{v}^*\) and \(\ttt{w}^*\)
    are multisets of vertices that may be terminal or non-terminal,
    and \ttt{x} and \ttt{y} are non-terminal.}
  \label{\pendingLabel}
\end{figure}

\begin{definition}[loop snipping]\label{D:loop-snipping}
  Let \(\ttt{a}\overrightarrow{\ttt{W}}\ttt{b}\) be a well-formed
  fragment of a bridge word for bridge \ttt{C}, and let the length of
  \(\overrightarrow{\ttt{W}}\) be even. Either or both of \ttt{a}
  and \ttt{b} may be terminal.

  A \emph{loop-snipping} move replaces \(\ttt{a}\overrightarrow{\ttt{W}}\ttt{b}\)
  with \ttt{ab}. This is only permitted if no bridge has any equator
  crossings between \ttt{a} and \ttt{b} proceeding either east or west from
  \ttt{a}. See \Fig{loop-snip} for the geometric interpretation.
\end{definition}

\nextLabel{fig:loop-snip}
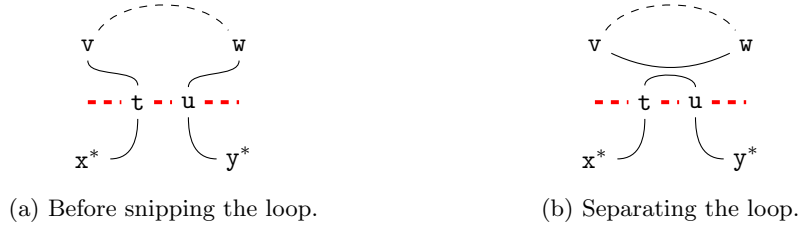
\begin{figure}
  \centering
  \begin{subfigure}[t]{0.47\textwidth}
    \centering
    \begin{tikzpicture}[yscale=0.75]
      \node (v) at (0, 2) {\ttt{v}};
      \node (w) at (2, 2) {\ttt{w}};
      \node (a) at (0.66, 1) {\ttt{t}};
      \node (b) at (1.33, 1) {\ttt{u}};
      \node (x) at (0, 0) {\(\ttt{x}^*\)};
      \node (y) at (2, 0) {\(\ttt{y}^*\)};
      \draw (v) to[out=270, in=90] (a);
      \draw (w) to[out=270, in=90] (b);
      \draw (a) to[out=270, in=0] (x);
      \draw (b) to[out=270, in=180] (y);
      \draw [local binding] (0, 1) -- (a) -- (b) -- (2, 1);
      \draw[dashed, rounded corners] (v) to[out=60, in=120] (w);
    \end{tikzpicture}
    \caption{Before snipping the loop.}
  \end{subfigure}
  \hfill
  \begin{subfigure}[t]{0.47\textwidth}
    \centering
    \begin{tikzpicture}[yscale=0.75]
      \node (v) at (0, 2) {\ttt{v}};
      \node (w) at (2, 2) {\ttt{w}};
      \node (a) at (0.66, 1) {\ttt{t}};
      \node (b) at (1.33, 1) {\ttt{u}};
      \node (x) at (0, 0) {\(\ttt{x}^*\)};
      \node (y) at (2, 0) {\(\ttt{y}^*\)};
      \draw (a) to[out=270, in=0] (x);
      \draw (b) to[out=270, in=180] (y);
      \draw (a) to[out=90, in=90] (b);
      \draw [local binding] (0, 1) -- (a) -- (b) -- (2, 1);
      \draw[dashed, rounded corners] (v) to[out=60, in=120] (w);
      \draw (v) to[out=330, in=210] (w);
    \end{tikzpicture}
    \caption{Separating the loop.}
  \end{subfigure}
  \caption{Loop snipping is strictly length-shortening. The upper loop
    in (b) is trivial and discarded. Because moving from (b) to (a)
    is knot composition with a trivial knot, moving from (a) to (b)
    does not change the knot type. 
    Vertex \ttt{v} may be \ttt{w}; vertex
    \ttt{t} may be \ttt{u} if \ttt{t} is not terminal; and
    \(\ttt{x}^*\) may be \(\ttt{y}^*\). Neither \(\ttt{x}^*\) nor
    \(\ttt{y}^*\) is required to exist.}
  \label{\pendingLabel}
\end{figure}

\section{Geometric Effects of Grammar Rules}
The grammar and the geometry are equivalent. We prove a few little lemmas to show this.

\begin{lemma}\label{L:oneway}
  Let \(K\) be a \(3\)-page bridge embedding. A bridge \ttt{A} that contains vertex
  \(\ttt{a}^{\circ}\) cannot be used for avoiding the underpass incident to
  \ottt{a}.
\end{lemma}
\begin{proof}
  The bridge-looping word around \ttt{A} must also include \(\ttt{a}^{\circ}\),
  so pushing crossings along \ttt{A} does not avoid the underpass incident to
  \ottt{a}.
\end{proof}

\begin{lemma}\label{L:twoways}
  Let \(K\) be a \(3\)-page bridge embedding.
  If \ottt{ab} exists, \ttt{A} contains \(\ttt{a}^{\circ}\), and
  \ttt{B} contains \(\ttt{b}^{\circ}\), then crossings cannot be removed from
  \ottt{ab} by the underpass avoidance move.
\end{lemma}
\begin{proof}
  Apply \Lemma{oneway} in each direction from \ottt{ab}.
\end{proof}

\begin{lemma}\label{L:minimal}
  If \(K\) is a \(3\)-page bridge embedding with no grammatical moves available, then \(K\)
  is a locally minimal bridge presentation.
\end{lemma}
\begin{proof}
  Suppose there are \(n\) underpasses in \(K\) and no moves are available. Then from
  \Lemma{twoways}, no crossings can be removed by an underpass avoidance rule and, because
  there are no grammatical moves available, no eager grammar rules will change the bridge words.

  From Tawn \cite{TawnCutComplex,TawnErratum}, every ambient isotopy that does not introduce a
  new bridge but changes crossings can be constructed as the composition of isotopies around a
  path between two vertices. For any single such isotopy in this complex, the difference between
  the diagram before the isotopy and after will form a loop. That is, if the diagram
  has a path \(p_0\) that exists before the isotopy but not after and a path \(p_1\)
  that exists after the isotopy but not before, then the isotopy can be chosen
  such that \(p_0 \cup p_1\) is a loop in pages \(S \cup N\).
  
  Suppose, for contradiction, that we can avoid a particular crossing 
  that obstructs a reduction in the number of bridges. We have \(n\) underpasses to avoid, and repeated application
  of \Lemma{oneway} therefore produces a chain of required bridge detours. Since there are
  only \(n\) bridges and no underpass can be closed during this chain, the chain must eventually
  repeat a bridge. Such repetition cannot complete an avoidance. That is, no
  loop that adds crossings allows an underpass removal. Therefore,
  there must be a crossing that can be relocated or removed without avoiding an
  underpass. However, this would imply the existence of a loop-snipping move,
  contrary to the assumption that no eager moves are available. By contradiction,
  no crossings can be relocated or removed in a way that makes an underpass
  disjoint from either of its adjacent bridges.

  Moving the vertices does not help. \(K\) is a bridge presentation, and Scharlemann and
  Tomova \cite[Lemma 3.1]{cancelingPairs} proved that one of the spheres in a bridge splitting may be
  held fixed during bridge perturbation and stabilization without preventing the discovery of
  canceling disc pairs. We may therefore hold the underpass sphere constant without affecting
  the available topological moves.

  Thus, if no grammatical move is available, then no sequence of geometric moves can change
  the number of bridges without either leaving bridge position or introducing a new bridge.
  The number of bridges is locally minimal.
\end{proof}

\section{Containing the Complexity}\label{S:complexity}
In this section, we show the size and time complexities of different elements of the process
in order to reach a worst-case complexity. We note, and will show, that a naive application without
any word-shortening moves could result in a system that has an exponential length relative to the number
of initial crossings. Despite this, by tracking arcs with multiplicity instead of sequences and
making use of the quotient move, we may contain these possibly \(2^{O(n)}\) crossings in \(O(n)\) space for each
of the \(n\) initial underpasses.

In this section, we show only the gross complexities. Implementation details such as indexing,
algorithms for bit arithmetic, and the specific methods chosen for tasks may result in
sharper or weaker bounds, which we will note where relevant.

On space complexity, we have a base complexity
\begin{align*}
  &O(n) &&\times & &O(n) &&\times & &O(n) & &=O(n^3)\text{.}\\
  &\text{Bridges} & &\ & &\text{Arcs} & &\ & &\text{Bits per arc count}
\end{align*}

This bound is sharp for the representation as arcs with counts if the maximum number of
crossings is \(2^{O(n)}\). However, some supplementary structures may be required during
calculations and some other proof may show a better bound on the number of crossings.
None of the auxiliary structures will make use of the expanded sentence, so
we will require space polynomial in the space required for the basic objects,
yielding a polynomial complexity overall.

On time complexity, we will show \(O(poly(n))\) to reduce a diagram with \(n\) initial crossings.
In a concrete implementation, the coefficient and exponent of the complexity will depend on the
complexities of some of the inner steps (e.g., finding loop cutting moves and performing
the normalization of internal repetition of letters). The implementation choices involve
a trade-off between the efficiency of data storage and the efficiency of data
indexing. For example, we may straightforwardly align arc bundles across the equator
in \(O(n)\) time if the time to access the counts of parallel edges is \(O(1)\) per edge,
but storing the counts in this way naively requires \(O(n^2)\) vertex positions (indexed
first by the position on the binding equator and second by the position of the other end on
the binding circle) despite only \(O(n)\) distinct chords being present in any page
(see \Lemma{rhssize}).

Because each underpass closure reduces the number of bridges by \(1\) until local minimality,
performing the bridge avoidance moves in aggregate by summing the counts of indistinguishable
arcs (rather than literally modifying the word) means there are only \(O(n)\) large steps.
Within each of these large steps, there are \(O(n)\) words to consider with \(O(n)\) potential
moves on each, and the addition and
multiplication of potentially \(O(n)\)-bit integers for each of \(O(n)\) pairings. Because the
arithmetic is on \(O(n)\)-bit words, we assume a conservative bound of \(O(n^2)\) time,
which gives a base time complexity of
\begin{align*}
  &O(n) & &\times & &O(n^2) & &\times & &O(n) & &\times & &O(n^2) & &=O(n^6)\text{.}\\
  &\text{Steps} & &\ & &\text{Simplifications} & &\ & &\text{Arcs} & &\ & &\text{Bit arithmetic}
\end{align*}
\begin{remark}
  As noted in \Section{construction}, it may take time \(O(n^2)\) to construct the \(3\)-page
  bridge embedding from a knot diagram, but the construction time is dominated by the calculation
  time, even when we later assume \(O(1)\) time for bit arithmetic in special cases.
\end{remark}

Despite this optimistic bound, normalization requires \(O(poly(n))\) time. Structurally,
this seems like it is somewhere around \(O(n^3)\) each time it occurs (see \Lemma{zipper}),
but we do not prove a better bound than polynomial.

\begin{lemma}
  At most \(n\) underpass closures are needed to reduce a diagram with an initial \(n\) bridges
  to a locally minimal bridge position.
\end{lemma}
\begin{proof}
  Each closure reduces the number of bridges by \(1\) and no move creates a new underpass.
\end{proof}

Straight-line programs (SLPs) are a fundamental concept of theoretical computer science.
An SLP is a context-free grammar that generates exactly one string. An SLP of size
\(O(n)\) will terminate in a string of length at most \(2^{O(n)}\). See Lohrey
\cite{LohreySLP} for a survey. We use only this result from the theory of SLPs, not any
optimization algorithms or other results.

\begin{lemma}\label{L:slpsize}
  Any sequence of metagrammar moves used here can generate a straight-line program that produces
  one string, the characterization of the resulting \(3\)-page bridge embedding, with
  length at most \(2^{O(n)}\).
\end{lemma}
\begin{proof}
  We use the construction from \Section{construction} to characterize a diagram with \(n\)
  crossings in \(O(n)\) distinct symbols. All terminal symbols will be from this set, although
  some symbols that appear in the initial encoding may not appear in the output. The metagrammar
  (see \Table{rules}) has a fixed size independent of \(n\), which may be applied to each
  of the \(O(n)\) vertex labels.
  
  At each underpass avoidance move, we use the grammar rule \(\ttt{a}^{\circ} \mapsto \ttt{A}'\)
  for some \(\ttt{a}^{\circ}\) in the initial set of symbols. Every other rule is
  length-reducing or length-neutral without using a new symbol on either side of the rule. This gives
  us \(O(n)\) grammar size for changes and \(O(n)\) grammar size for the initial diagram,
  which is an \(O(n)\) grammar for the whole system. We need \Lemma{rhssize} here to
  be sure that only \(O(n)\) objects (whether bridges and their bridge-looping words or arcs)
  need to appear on the right-hand side at any stage.
  
  The metagrammar is not itself a context-free grammar, relying on context in some places and
  conditionality in others. However, we are interested in a worst-case analysis of the total
  output size, so the conditionality is only a problem because it would, if unaddressed,
  put what may be a terminal symbol on the left-hand-side of the production rule.
  To avoid this, we never formally create an uninvoked rule, despite the metagrammar telling
  us how we \emph{could} construct it. In spirit, we construct a grammar after the
  fact to evaluate the potential size of the final sentence.
  
  \begin{itemize}
    \item Normalization is not context-free, but it is strictly length-reducing.
    \item Quotient is conditional but context-free because the left-hand side is eliminated everywhere
      when invoked, so the left-hand side is a non-terminal symbol.
    \item Underpass closure is conditional but context-free and length-reducing.
    \item Underpass avoidance is conditional but context-free.
    \item Loop snipping is conditional and not context-free, but strictly length-reducing.
  \end{itemize}

  The constructed grammar is a directed acyclic graph because we construct the rules exactly
  when we will apply them, which avoids any forward-reference to rules and thus any cycles of rule
  application. We can index these rules, if desired, by writing a new rule and using the count
  of underpasses already closed as an index for any replacements if explicit grammar rules
  are desired. Every rule that can increase the sentence length is context-free.
  This is an SLP of size \(O(n)\).

  From \Lemma{minimal}, after all moves, the outcome is a sentence characterizing
  a \(3\)-page bridge embedding of a diagram as a locally minimal bridge presentation.
\end{proof}
\begin{remark}
  Because we construct our grammar retroactively, time-complexity results from the
  theory of SLPs cannot be directly applied, but we can analyze the time-complexity separately.
  We do our time-cost accounting in arc form after doing this space-cost accounting
  in grammar form by appeal to an SLP.
\end{remark}

\begin{lemma}\label{L:rhssize}
  At most \(O(n)\) arcs are needed across the entire sentence of an initially \(n\)-crossing knot
  diagram.
\end{lemma}
\begin{proof}
  From the non-crossing condition and the initial number of vertices, this is, at worst, two
  maximal outerplanar graphs, one on page \(S\) and another on page \(N\). Each of these
  has at most \(4n-3\) arcs, so the sentence encodes at most \(8n-6\) arcs for bridges and
  \(n\) arcs for underpasses. From \Lemma{quotient}, we never need new vertices
  that remain beyond one quotient action, and the only distinct arcs that may be added,
  rather than just replacing an existing arc, are a single arc added during underpass
  avoidance, and all arcs must still fit within the \(9m - 6\) distinct arcs when
  \(m\) bridges remain.
\end{proof}

For the above proof, we treat the multiplicity of indistinguishable arcs as irrelevant.

\section{Sentence Perspective and Arc Perspective}
In the proof of \Theorem{main}, we use a formal grammar on words but mention that
this corresponds directly to a set of arcs. This grammar and its pieces are relatively
unknown, so in this section, we describe how to convert between the two styles, which
is also necessary for compressed output of final diagrams as edge-weighted graphs, and
how the two perspectives simplify or complicate different steps.

We make a set of granular definitions here to clarify the relation between the
two ways of considering the geometry of the \(3\)-page bridge embeddings. For all definitions,
assume that \(n\) is the number of bridges.

\begin{definition}[vertex label]
  There are \(2n\) endpoints of the \(n\) bridges, and a vertex label is a member of
  a set of \(2n\) distinct labels.
\end{definition}

We may use any labels that are distinct, and in figures we have used both numerals and
lower-case roman letters. The labels are atomic and have no inherent meaning, though
a convention such as ``labels are in order of knot traversal'' can be helpful for reading
diagrams visually. This is the convention we have followed here.

\begin{definition}[circle word]
  In the grammar, given a list of \(2n\) distinct vertex labels, a \emph{circle word}
  is a semicolon followed by an ordered list of these symbols. Using this word, each vertex
  label is associated to the natural number such that the \(i\)th label has value \(2i - 2\)
  for \(i = 1, 2, \ldots, 2n\).
\end{definition}

\begin{definition}[cyclic order, east, west]
  We associate half of the \(4n\) vertices to vertex labels in the circle word. These are
  the terminal vertices. In the grammar, this means they are either the first or last letter/label
  in a bridge word. In the embedding, this means they have exactly one edge incident to the
  vertex on page \(U\).

  The \emph{cyclic order} is given by the abelian group \((\mathbb{Z}_{4n}, +)\), with
  \(\sigma(x) = y\) meaning \(y \equiv x + 1 \mod 4n\) among all vertices and
  \(\sigma(x) = y\) meaning \(y \equiv x + 2 \mod 4n\) among all terminal vertices.

  Letting \(x_{E} = \sigma(x)\) and \(\sigma(x_{W}) = x\), we call \(x_{E}\) the vertex \emph{east}
  of \(x\) and call \(x_{W}\) the vertex \emph{west} of \(x\). 
\end{definition}

Because of \Lemma{quotient}, non-terminal vertex labels are interchangeable up to the
nearest terminal vertex moving west, and we use that label for them. That is, if \(x\) is a
label, we will have two vertices with that label, one terminal and one to the east of that
terminal vertex. In the grammar, we can directly identify between them because of the position
in a bridge word, so this does not cause ambiguity.

\begin{definition}[bridge word in grammar]
  A \emph{bridge word} begins with a \ttt{+} or \ttt{-} symbol and is followed by at least
  two vertex labels. The first and last vertex labels cannot be the same because this would form
  a loop beginning and ending at the same terminal vertex, which is not a bridge.
\end{definition}

Note that we ignore the case where there is exactly one bridge in the diagram because when
this happens, we declare the diagram trivial and finish.

Given the circle word, we know the arrangement of vertex labels, which we may treat as
either a linear order or as a cyclic order. The cyclic order is more convenient because it
makes every arc a chord and eliminates the irrelevant distinction between the two kinds of
non-interleaving pairs of arcs (\emph{nested} or \emph{unnested} in the linear order).

In the pages of the embedding, every bridge is a path along arcs in \(S \cup N\), and the
bridge word says exactly how these are embedded up to ambient isotopy.
Let us define an arc in our context.

\begin{definition}[arc]
  An \emph{arc} is an unordered pair of vertex positions together with a page (\(U\), \(S\), or \(N\)).
\end{definition}

We can immediately take care of underpasses, which we will not consider again until later. Note
that because the vertex pair is unordered, \((x, y, P) = (y, x, P)\) for any labels
\(x\) and \(y\) and page \(P\). We explicitly permit \((x, x, P)\) as an arc to be dealt with
by the normalization rules. (See \Lemma{zipper}.)

\begin{definition}[underpass embedding]
  An \emph{underpass} in the \(3\)-page bridge embedding from the viewpoint of arcs is an arc
  \((a, b, U)\), where \(a\) and \(b\) are terminal and \(a_{E} = b\) or \(b_{E} = a\) among
  terminal vertices.
\end{definition}

This says exactly what the grammar construction did: an underpass lives on page \(U\) and is an
arc between adjacent terminal vertices (ignoring intervening non-terminal vertices),
using the east/west meaning for adjacency. The direction is not well-defined in exactly one
case: when there is exactly one bridge. In that case, the diagram is trivial.

Now we decompose a bridge word and show its representation as arcs.

If the bridge word begins with \ttt{+} (resp., \ttt{-}), then the first named arc belongs in
page \(N\) (resp., \(S\)). After that, the next named arc belongs in the opposite page, switching
pages between \(S\) and \(N\) at each letter and using the next two letters as the arc ends.
For example, if the complete bridge word is \ttt{+abad}, then we have arc \((\ottt{a}, \ttt{b}, N)\),
arc \((\ttt{b}, \ttt{a}, S)\), and arc \((\ttt{a}, \ottt{d}, N)\) appearing
in that order. The order of the arcs in the geometry can be reconstructed, so in terms of arcs,
each bridge word encodes a multiset of arcs with their multiplicities. Note that the letter
\ttt{a} appears in the bridge word twice, once as a terminal vertex and once as the non-terminal
vertex immediately east of that terminal vertex.

For a bridge word of length \(k\), there are \(k - 1\) arcs encoded. Whether the non-terminal
vertices \ttt{b} and \ttt{a} imply a knot crossing can be determined from the underpass
arcs. If \ottt{a} is (resp., is not) on the west end of an underpass,
then the non-terminal vertex with label \ttt{a} does (resp., does not) indicate a knot crossing.
In the metagrammar, we write this as \(\ttt{a}^\circ\) (crossing) or \(\ttt{a}'\) (non-crossing),
and we can know immediately from the circle word what the symbols of those two vertices will be.
This geometric interpretation also makes immediately clear why we can reverse the direction in
which a bridge is written and why it might or might not require switching the sign, depending on
only the parity of the length of the word. (Odd length requires switching; even length does not.)

The definition of a bridge-looping word and how to use it in underpass avoidance can seem a
bit arbitrary in the grammar, although every outcome is of the same length. The case is
straightforward in arc form.

\begin{definition}[bridge-looping word]\label{D:bridge-loop-arcs}
  As a set of arcs, given a bridge \ttt{A} with ends \(\ottt{a}\) and
  \(\ottt{b}\) with initial arc \((\ottt{a}, \ttt{r}, X)\),
  terminal arc \((\ottt{b}, \ttt{s}, Y)\), and multiset \(M\) of non-terminal arcs,
  the bridge-looping set of arcs is found by the following steps.
  \begin{enumerate}
    \item Begin with \(M\).
    \item Double all of the multiplicities of \(M\).
    \item Add \((\ottt{a}_{W}, \ottt{a}_{E}, \neg X)\) where
      \(\neg X\) is the page in \(\{S, N\}\) that is not \(X\).
    \item Add \((\ottt{b}_{W}, \ottt{b}_{E}, \neg Y)\).
    \item Add \((\ottt{a}_E, r, X)\) and \((\ottt{a}_W, r, X)\).
    \item Add \((\ottt{b}_E, s, Y)\) and \((\ottt{b}_W, s, Y)\).
  \end{enumerate}
\end{definition}

Geometrically, this puts caps on the two ends of the bridge and surrounds its internal arcs.
All of the complication in the grammar comes from trying to linearize a circle and orient a
collection of unoriented arcs. We note that if the bridge word is exactly two letters \emph{and}
the two letters are adjacent as terminal vertices, then this will include a sequence of
three arcs that will have a trivial arc in the middle that will be removed via normalization.
Geometrically, this is easy: the two caps will have one shared endpoint.

\begin{definition}[Underpass avoidance]
  From the arc perspective of the bridge-looping word and bridge words, underpass avoidance no
  longer requires case-by-case, letter-by-letter analysis. Instead, we do the following for each
  bridge that crosses the underpass to be avoided. Let us suppose without loss of generality that
  the presence of label \ttt{a} in the middle of a bridge word indicates a crossing of the underpass
  with west end \(\ottt{a}\) and that we will avoid it via bridge \ttt{A}.
  \begin{enumerate}
    \item Get the multiset of the bridge-looping word for \ttt{A}.
    \item Remove arcs \((\ttt{a}, x, N)\) and \((\ttt{a}, y, S)\).
    \item Multiply all arc multiplicities by the number of crossings being removed.
    \item In each bridge that crosses at \ttt{a}, replace \((\ttt{a}, w, N)\) with
      \((x, w, N)\) and replace \((\ttt{a}, v, S)\) with \((y, v, S)\).
  \end{enumerate}
\end{definition}

Notice that we can do this in bulk. Geometrically, we do not need to operate one bridge at a time
for the avoidance in this form, but caution is nevertheless needed. Although the summed arcs across
all bridges will encode the diagram correctly, recovering a specific bridge after mixing the
arcs of multiple bridges may require full expansion of the diagram, and we have only an
exponential bound for the actual number of arcs when not compressing by multiplicity.
Because we need to know the bridges individually to find their bridge-looping words and
perform normalizations, we keep them separate during underpass avoidance.

Twice now, we have mentioned eliminating trivial arcs, by which we mean eliminating arcs
between binding-adjacent vertices. This is just the geometric equivalent of the normalizing
moves on the grammar. It also gives a practical bound on the number of normalizations that might
be performed in succession. Each move from \(3\) arcs to \(1\) arc (as in underpass avoidance)
occurs with the quotienting of multiple labels, and the number of labels decreases monotonically as
underpasses are closed. Each move from \(2\) arcs to \(1\) arc is the simplification of the end of
a bridge, and there are only two distinct non-terminal points that can be involved
(the vertices immediately east and west of the terminal vertex), so this can occur at most
\(4\) times per bridge isotopy, twice at each end; otherwise, there is a cycle on those
non-terminal vertices that would make the bridge self-intersecting. Any additional instances
would be handled by the rule for repeated-letter normalization.

Algorithmically, because arcs that are literally \((x, x, P)\) are not safe, we use
arc-bundle expansion to handle normalization and underpass avoidance. See \Fig{avoid} for the
technique used in underpass avoidance. For discharging trivial arcs as a kind of
diagram-altering quotient, we use the technique shown in \Fig{zipper}.
\Lemma{zipper} proves that this terminates in \(O(poly(n))\) steps despite not having a better
bound on weights than exponential.

\begin{lemma}\label{L:zipper}
  Two adjacent non-terminal vertices in a non-crossing path may be quotiented together
  in \(O(poly(n))\) steps.
\end{lemma}
\begin{proof}
  Let the adjacent non-terminal vertices be \(v_1\) and \(v_2\). If there is no edge between
  them, then \Lemma{quotient} completes the proof.
  
  Suppose there is an arc \(v_1v_2\) with weight \(w(v_1v_2)\) and a
  vertex \(v_3\) with edges \(v_1v_3\) and \(v_2v_3\). If \(w(v_1v_2)\) is greater than
  \(\min(w(v_1v_3), w(v_2v_3))\), then the path is self-intersecting, contrary to the
  non-crossing assumption. We split this vertex \(v_3\) into two copies and shift the
  edge weights as shown in \Fig{zipper}. This eliminates the triangle by creating
  \(0\) or \(1\) triangles elsewhere.

  In contrast, if there is no such \(v_3\), then we can choose an arbitrary \(v_3\)
  with no edge \(v_2v_3\), so \(w(v_1v_2) = \min(w(v_1v_3), w(v_2v_3)) = 0\),
  and \Lemma{quotient} lets us stop.

  By repeated application, we move the quotient along the path until we discharge the edges.
  This process must terminate because it would otherwise form a cycle, which would correspond to
  a cycle in the original path, contrary to assumption. By \Lemma{rhssize}, there
  are only \(O(n)\) distinct arcs to follow during this process.
\end{proof}

Although this process, which may be viewed as the action of a zipper moving a bigon through
the diagram, takes only \(O(n)\) steps, the bookkeeping required for doing so is complicated
enough that selection of algorithms for splitting the arc bundles, managing the bit arithmetic,
and following the diagram is non-trivial, so only \(O(poly(n))\) is assumed here.

In addition to the above proof, we note that each geometric change in the arcs after the
vertex quotient corresponds to the elimination of \(w(v_1v_2)\) letter-doubles in the
sentence. In the grammar, this is the internal normalization move. This is guaranteed to
terminate correctly in grammar terms, and \Lemma{zipper} is about counting how many steps
this can take. The grammar rule \(\ttt{aa} \to \emptyset\) is not context-free, but we may verify
the count of \Lemma{zipper} by examining the grammar, noting that eliminating \(x\)
doublings eliminates \(2x\) letters and lowers the total of the edge weights by \(2x\).
The elimination of one set of doubled-letters may expose a different set of doubled-letters,
but will not expose more than one distinct letter at a time, and no letter pair will repeat in the
normalization.

\nextLabel{fig:zipper}
\begin{figure}
\centering
\begin{subfigure}{0.47\textwidth}
  \centering
  \begin{tikzpicture}[every node/.append style={fill=diagrambg, inner sep=2pt}]
    \def\R{2}
    \coordinate (a) at (60:\R);
    \coordinate (b) at (120:\R);
    \coordinate (c) at (180:\R);
    \coordinate (e) at (270:\R);
    \coordinate (d) at (0:\R);
    \draw (a) -- node[midway, above] {\(w(ab)\)} (b);
    \draw (b) -- node[midway, above, sloped] {\(w(bc)\)} (c);
    \draw (a) -- node[midway, above, sloped] {\(w(ad)\)} (d);
    \draw (b) -- node[midway, below, sloped] {\(w(be)\)} (e);
    \draw (a) -- node[midway, below, sloped] {\(w(ae)\)} (e);
    \draw (c) -- node[midway, below, sloped] {\(w(ce)\)} (e);
    \draw (d) -- node[midway, below, sloped] {\(w(de)\)} (e);
    \node at (a) {\ttt{a}};
    \node at (b) {\ttt{b}};
    \node at (c) {\(\ttt{c}^*\)};
    \node at (e) {\ttt{e}};
    \node at (d) {\(\ttt{d}^*\)};
  \end{tikzpicture}
  \caption{Normalizing before quotient. The weight \(w(ab) = \min(w(ae), w(be))\), and
    we assume wlog that it is \(w(ae)\).}
\end{subfigure}
\hfill
\begin{subfigure}{0.47\textwidth}
  \centering
  \begin{tikzpicture}[every node/.append style={fill=diagrambg, inner sep=2pt}]
    \def\R{2}
    \coordinate (a) at (90:\R);
    \coordinate (c) at (180:\R);
    \coordinate (e1) at (240:\R);
    \coordinate (e2) at (300:\R);
    \coordinate (d) at (0:\R);
    \draw (a) -- node[midway, above, sloped] {\(w(bc)\)} (c);
    \draw (a) -- node[midway, above, sloped] {\(w(ad)\)} (d);
    \draw (a) -- node[pos=0.55, above, sloped] {\(w(be)-w(ae)\)} (e1);
    \draw (c) -- node[midway, below, sloped] {\(w(ce)\)} (e1);
    \draw (d) -- node[midway, below, sloped] {\(w(de)\)} (e2);
    \draw (e1) -- node[midway, below, sloped] {\(w(ae)\)} (e2);
    \node at (a) {\ttt{a}};
    \node at (c) {\(\ttt{c}^*\)};
    \node at (e1) {\(\ttt{e}_1\)};
    \node at (e2) {\(\ttt{e}_2\)};
    \node at (d) {\(\ttt{d}^*\)};
  \end{tikzpicture}
  \caption{Merging \ttt{a} and \ttt{b}, transferring the residual weight.}
\end{subfigure}
\caption{Normalization of trivial arcs via quotient moves. The vertices \ttt{a} and
  \ttt{b} are adjacent and non-terminal, and \ttt{e} is non-terminal.
  Labels \(\ttt{c}^*\) and \(\ttt{d}^*\) indicate sets of vertices
  (perhaps empty or singleton). The total weight is reduced by \(2w(ab) = 2w(ae)\).}\label{\pendingLabel}
\end{figure}
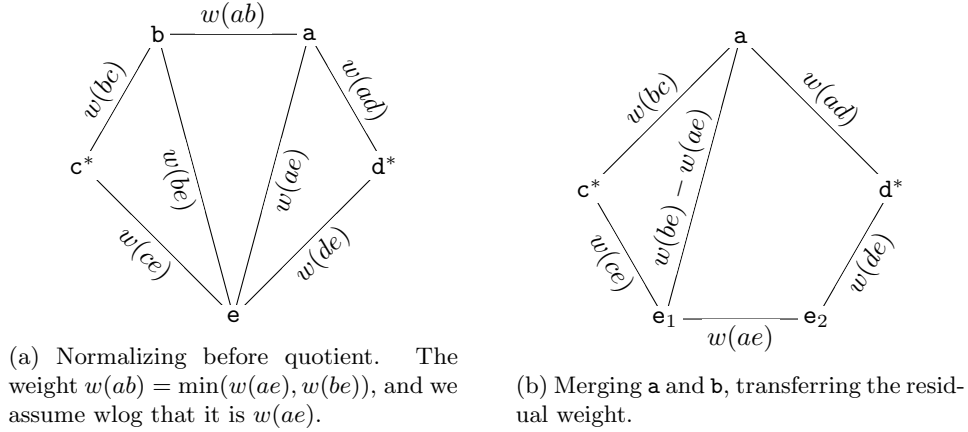

\subsection{Finding and Performing Loop Snipping}
Loop snipping operates similarly but not identically. On an arc-by-arc basis, finding potential
loop-snipping moves is straightforward, but we want to avoid expansion of arc bundles into
multiple copies of arcs. As with underpass-avoidance moves, we operate in bulk. In contrast with
underpass-avoidance moves, care is needed on ordering because the availability of the move
is conditioned on what other bridges are doing.

Fortunately, the non-crossing condition helps us here: we can, without sorting out the exact
nesting order of bridges, determine whether loop-snipping can be done simultaneously for a
set of bridges on a vertex-pair and page basis.

\begin{lemma}\label{L:loop-snipping}
  Let \(\{B_1, B_2, \ldots, B_n\}\) be the set of bridges in a \(3\)-page bridge presentation
  with punctures \(v_1, v_2, \ldots, v_{4n}\). For any pair of vertices \((v_a, v_b)\) and page
  \(X \in \{S, N\}\), the set of bridges with an eligible loop-snipping move to connect
  \(v_a\) and \(v_b\) in page \(X\) may be checked without unpacking the bridge words and,
  when a bridge can be shortened by this loop-snipping move, the multiplicity is also recoverable
  without unpacking the bridge words.
\end{lemma}
\begin{proof}
  If no loop-snipping move exists for the potential arc \(v_a, v_b, X\), then there is nothing
  to prove, so assume without loss of generality that some loop-snipping move exists that
  replaces a longer sequence with \((v_a, v_b, X)\).

  If either \(v_a\) or \(v_b\) is terminal, the arc to be replaced has multiplicity \(1\), so
  checking does not require geometric expansion. Suppose neither \(v_a\) nor \(v_b\) is terminal.

  Without loss of generality, assume that bridge \(B_1\) is a bridge to be shortened by loop
  snipping. By assumption, the arcs \((v_a, v_x, X)\) and \((v_b, v_y, X)\) in \(B_1\) can
  participate in loop-snipping and may do so in a multiplicity that can be determined from
  arc-bundle arithmetic. It is possible, however, if there is another bridge \(B_2\) that
  is also eligible for loop-snipping except for the possibility that \(B_1\) interferes,
  then we cannot be certain without full geometric unpacking which potential loop is
  innermost.

  If we perform all of the loop-snipping moves for \((v_a, v_b, X)\) simultaneously, this is
  not a problem. In principle, but importantly \emph{not} in practice, we choose the innermost
  loop-snipping move, perform it, and repeat until none remain. We can find and perform the
  move on one bridge and can add one more bridge to any subset for which we can do this, so
  by induction we can do it for any subset of (eligible) bridges.

  In the case where \(v_a = v_b\), we temporarily split \(v_a\) into two distinct vertices,
  perform the loop-snipping moves, and then quotient the two temporary vertices together,
  using the zipper technique of \Lemma{zipper}.

  The zipper technique applies generally. Although the normalization rule requires the
  vertices used in \Lemma{zipper} to be adjacent and non-terminal, this restriction is used
  in verifying the availability of moves across the diagram, not in performing the moves.
  On a bridge-by-bridge basis, we perform loop-snipping in the same way that we perform
  interior normalization.
\end{proof}

\section{Finding Locally Minimal Bridge Presentations}
\begin{theorem}\label{T:main}
  Given a knot diagram \(K\) with \(n\) crossings, a locally minimal bridge
  presentation of the same knot type can be found in polynomial time and
  space. The sequence of choices determining this presentation can be encoded
  as a certificate of length \(O(n)\) if only the bridge count is needed or
  \(O(n^2)\) for the entire diagram.
\end{theorem}
\begin{proof}
  We first convert the diagram to a \(3\)-page bridge embedding. From Bekos, Gronemann,
  and Raftopoulou \cite{Hamiltonian}, we can do this in time \(O(n^2)\).
  From \Lemma{firstLength}, the length of the sentence describing this embedding is
  \(O(n)\).

  From \Lemma{rhssize}, the number of distinct arcs in the diagram is \(O(n)\).
  Because the grammar forms a straight-line program (see \Lemma{slpsize}) with
  maximum length \(2^{O(n)}\) and each crossing is represented by a letter in the final
  full characterization, the largest arc multiplicity possible is \(2^{O(n)}\), which
  requires \(O(n)\) bits to store. We may give the full diagram in space \(O(n^2)\)
  by summing all arcs together across the remaining bridges, at the cost of making the
  per-bridge information take (naively) exponential time to reconstruct.

  On time, we require arithmetic on \(O(n)\)-bit numbers, which takes time
  \(O(n^2)\) under typical complexity rules.

  Between each bridge reduction, we must perform \(1\) underpass avoidance for each crossing being
  moved off of the underpass being avoided, after which the underpass is closed in constant time.
  Performing the avoidance is done by manipulating the arc counts. Specifically,
  replacing an arc repeated \(k\) times in \(m\) places, with all repetition counts being
  \(O(n)\)-bit numbers, can be effected by multiplying the counts together and then summing them.
  The cost \(O(n^2)\) for multiplication dominates this time. We track these counts separately
  for each bridge.

  Like underpass avoidance, the eager moves can be done via the manipulation of arc counts
  (\Lemma{zipper}, \Lemma{loop-snipping}), rather than direct manipulation of the
  characterization sequence. Normalization of non-terminal
  arcs to remove trivial arcs can require chasing the unwinding, but only up to indistinguishable
  arcs. From \Lemma{rhssize}, there are at most \(O(n)\) distinct arcs involved for \(O(n)\),
  bridges. For loop-snipping moves, there are only \(O(n^2)\) pairings of arcs to check
  at each stage, and performing the snipping move is bit arithmetic, rather than geometric
  unwinding. The bookkeeping for doing these moves has only been shown to be polynomial,
  giving time and space bound \(O(poly(n))\).

  By \Lemma{minimal}, the sequence of moves produces a bridge presentation that is locally
  minimal.

  The certificate contains the initial \(3\)-page bridge embedding, the underpass avoidances
  in order and direction, and either the number of bridges or the resulting diagram. The
  certificate therefore has length \(O(n)\) for only the bridge count and length \(O(n^2)\)
  for the resulting bridge presentation.
\end{proof}

\begin{corollary}
  Given a knot diagram \(K\) with \(n\) crossings, it can be decided in polynomial time and
  space whether \(K\) is trivial, with a certificate of size \(O(n)\).
\end{corollary}
\begin{proof}
  From Otal \cite{Otal1982}, a locally minimal bridge presentation of the trivial knot has
  exactly \(1\) bridge. From \Theorem{main}, a locally minimal bridge count may be
  found for \(K\) in polynomial time and space. \(K\) is trivial if and only if
  the count is \(1\), and the certificate that omits the diagram is sufficient for
  unknot recognition.
\end{proof}

\appendix

\section{Minimizing a Trefoil}
We minimize a trefoil, starting with the sentence \ttt{+142+304+520;014523}. (See \Fig{trefoil}.)

Apply underpass avoidance to avoid underpass \ottt{23} via bridge \ttt{3}.
Notice that \(\ttt{3}^{\circ} = \ttt{2}\), contained in bridge \ttt{5}. The
bridge-looping word around bridge \ttt{3} is \ttt{+2014032}. In this case, it is
small enough to find this by inspection: the interior of bridge \ttt{3} is \ttt{0},
with end pair \((\ttt{2}, \ttt{3})\) next to \ottt{3} and
end pair \((\ttt{1}, \ttt{4})\) next to \ottt{4}. To avoid \ttt{2},
we use the internal \ttt{01403}. If \ttt{2} appeared first in page \(S\), we would
need to use \ttt{30410} instead.

Applying the rule, \(\ttt{+520} \mapsto \ttt{+5014030}\), and then normalization gives
\ttt{+501400} and then \ttt{+50140} since both \ttt{0} and \ttt{3} are adjacent
to \ottt{0}. Now, we have the sentence \ttt{+142+304+50140;014523}. We close underpass
\ottt{23} to combine the bridge words for bridges \ttt{1} and \ttt{3},
getting the word \ttt{+14304}. We quotient the now non-terminal points,
the crossing vertex \ttt{2}, and the non-crossing vertex \ttt{3},
sending them all to the non-crossing vertex \ttt{5}, so the word becomes \ttt{+14504}.

The non-terminal points \ttt{0} and \ttt{4} are the underpass-crossing points,
and every bridge has its own undercrossing point. The resulting sentence
is \ttt{+14504+50140;0145}. See \Fig{trefoil2}, which
shows the pages separately in compressed form (all edges have weight \(1\) here) and, in
\Fig{trefoil2-d}, joined and fully expanded.

\nextLabel{fig:trefoil2}
\begin{figure}
  \centering
  \begin{subfigure}[t]{0.47\textwidth}
    \centering
    \begin{tikzpicture}[every node/.append style={fill=diagrambg, inner sep=2pt}]
        \coordinate (v0) at (0:1);
        \coordinate (v0e) at (45:1);
        \coordinate (v1) at (90:1);
        \coordinate (v1e) at (135:1);
        \coordinate (v4) at (180:1);
        \coordinate (v4e) at (225:1);
        \coordinate (v5) at (270:1);
        \coordinate (v5e) at (315:1);
        \draw[local binding] (0, 0) circle[radius=1];
        \draw[local strand] (v5e) -- (v4e);
        \draw[local strand] (v4e) -- (v0);
        \draw[local strand] (v0e) -- (v4);
        \draw[local strand] (v0e) -- (v1e);
        \node at (v0) {\ttt{0}};
        \node at (v0e) {\(\bullet\)};
        \node at (v1) {\ttt{1}};
        \node at (v1e) {\(\circ\)};
        \node at (v4) {\ttt{4}};
        \node at (v4e) {\(\bullet\)};
        \node at (v5) {\ttt{5}};
        \node at (v5e) {\(\circ\)};
    \end{tikzpicture}
    \caption{Page \(S\)}
  \end{subfigure}
  \begin{subfigure}[t]{0.47\textwidth}
    \centering
    \begin{tikzpicture}[every node/.append style={fill=diagrambg, inner sep=2pt}]
        \coordinate (v0) at (0:1);
        \coordinate (v0e) at (45:1);
        \coordinate (v1) at (90:1);
        \coordinate (v1e) at (135:1);
        \coordinate (v4) at (180:1);
        \coordinate (v4e) at (225:1);
        \coordinate (v5) at (270:1);
        \coordinate (v5e) at (315:1);
        \draw[local binding] (0, 0) circle[radius=1];
        \draw[local strand] (v5e) -- (v0e);
        \draw[local strand] (v5) -- (v0e);
        \draw[local strand] (v4e) -- (v1);
        \draw[local strand] (v4e) -- (v1e);
        \node at (v0) {\ttt{0}};
        \node at (v0e) {\(\bullet\)};
        \node at (v1) {\ttt{1}};
        \node at (v1e) {\(\circ\)};
        \node at (v4) {\ttt{4}};
        \node at (v4e) {\(\bullet\)};
        \node at (v5) {\ttt{5}};
        \node at (v5e) {\(\circ\)};
    \end{tikzpicture}
    \caption{Page \(N\)}
  \end{subfigure}
  \begin{subfigure}[t]{0.47\textwidth}
    \centering
    \begin{tikzpicture}[every node/.append style={fill=diagrambg, inner sep=2pt}]
        \def\aZero{0} \def\aOne{90} \def\aFour{180} \def\aFive{270}
        \def\rN{1} \def\rOut{1.25} \def\rOutA{1.5} \def\rOutC{2}
        \draw[draw=none, use as bounding box] (-\rN, -\rOutA) rectangle (\rN, \rOutC);
        \coordinate (v0) at (\aZero:1);
        \coordinate (v1) at (\aOne:1);
        \coordinate (v4) at (\aFour:1);
        \coordinate (v5) at (\aFive:1);
        \draw[local binding] (0, 0) circle[radius=1];
        \draw[local strand] (v0) arc[start angle=\aZero, end angle=\aOne, radius=1];
        \draw[local strand] (v4) arc[start angle=\aFour, end angle=\aFive, radius=1];
        \node at (v0) {\ttt{0}};
        \node at (v1) {\ttt{1}};
        \node at (v4) {\ttt{4}};
        \node at (v5) {\ttt{5}};
    \end{tikzpicture}
    \caption{Page \(U\)}
  \end{subfigure}
  \begin{subfigure}[t]{0.47\textwidth}
    \centering
    \begin{tikzpicture}[every node/.append style={fill=diagrambg, inner sep=2pt}]
      \def\aZero{0} \def\aZeroA{36} \def\aZeroB{72} \def\aOne{108} \def\aOneA{144}
      \def\aFour{180} \def\aFourA{216} \def\aFourB{252} \def\aFive{288} \def\aFiveA{324}
      \def\rIn{0.75} \def\rN{1} \def\rOut{1.25} \def\rOutA{1.5} \def\rOutB{1.75} \def\rOutC{2}
      \draw[draw=none, use as bounding box] (-\rOutB, -\rOutA) rectangle (\rOutB, \rOutC);
      \draw[local binding] (0, 0) circle[radius=1];
      \coordinate (v0) at (\aZero:\rN);
      \coordinate (v1) at (\aOne:\rN);
      \coordinate (v4) at (\aFour:\rN);
      \coordinate (v5) at (\aFive:\rN);
      \draw[local strand] (v0) arc[start angle=\aZero, end angle=\aOne, radius=\rN];
      \draw[local strand] (v4) arc[start angle=\aFour, end angle=\aFive, radius=\rN];
      \draw[local strand] (v0) -- (\aZero:\rOutB) arc[start angle=\aZero, end angle=\aFourA, radius=\rOutB]
        -- (\aFourA:\rIn) -- (\aOneA:\rIn)
        -- (\aOneA:\rOut) arc[start angle=\aOneA, end angle=\aZeroB, radius=\rOut]
        -- (\aZeroB:\rIn) -- (\aFive:\rIn) -- (v5);
      \draw[local strand] (v1) -- (\aFourB:\rIn) -- (\aFourB:\rOutA) arc[start angle=\aFourB, end angle=\aFiveA, radius=\rOutA]
        -- (\aFiveA:\rIn) -- (\aZeroA:\rIn) -- (\aZeroA:\rOutA) arc[start angle=\aZeroA, end angle=\aFour, radius=\rOutA] -- (v4);
      \node at (v0) {\ttt{0}};
      \node at (v1) {\ttt{1}};
      \node at (v4) {\ttt{4}};
      \node at (v5) {\ttt{5}};
    \end{tikzpicture}
    \caption{All pages together}
    \label{\pendingLabel-d}
  \end{subfigure}
  \caption{Pages of the \(3\)-page embedding of the \(2\)-bridge presentation of a trefoil with
    sentence \ttt{+14504+50140;0145}. Vertex labels \ttt{0}, \ttt{1}, \ttt{4},
    and \ttt{5} indicate terminal vertices. Vertex labels \(\bullet\) and \(\circ\)
    indicate non-terminal vertices that are crossing (\(\bullet\)) or non-crossing (\(\circ\))
    in the knot diagram.}
  \label{\pendingLabel}
\end{figure}
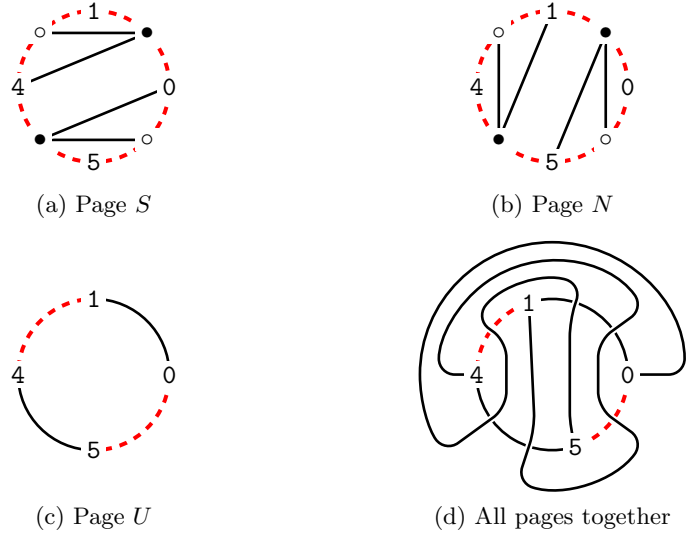

\section{Finding the Right Bridge-Looping Word}
The main text asserts that finding the correct part of the bridge-looping word to use for
underpass avoidance depends on only terminal vertex positions and word length. We now
prove that.

\begin{lemma}
  The correct order between \(\ttt{a}^{\circ}\) and \(\ttt{a}'\) switches
  when the bridge word has even length.
\end{lemma}
\begin{proof}
  An even-length word implies an odd number of arcs. Suppose we begin immediately west of one
  end of a chord and follow the chord. At the other end of the chord, we will be immediately
  east of the far end of the chord. Induction on parity completes the proof.
\end{proof}

\begin{lemma}
  The correct side to start avoidance depends on the orientation of the two ends of the underpass being
  avoided.
\end{lemma}
\begin{proof}
  If the bridge being taken is west (resp., east) of the crossing point, then we need to start the
  avoidance west (resp., east) of that terminal endpoint. We can then use the parity of the bridge
  being used in order to determine whether to initially pass east or west of the far end of the
  bridge. We will pass east and west in succession, so we are deciding only the order here.
\end{proof}

\begin{lemma}
  Whether to reverse the bridge-looping word depends on page agreement between the arc being
  re-routed and the arc being used for re-routing.
\end{lemma}
\begin{proof}
  Because we do not want to cross the equatorial binding, we reverse the looping word if and only
  if that makes the initial page of the looping word agree with the page of the arc just before
  the crossing arc. See \Fig{avoid}.
\end{proof}

\section{Reducing Knot \(6_3\) to a \(2\)-Bridge Presentation}
Knot \(6_3\) with Dowker code \(4,\ 8,\ 10,\ 2,\ 12,\ 6\) has a \(3\)-page bridge embedding
characterized by the sentence \ttt{-bhc-dbe-fkg-hdi-jek+lja;efkljidchgba}. We pick this
knot because it is small enough to trace the moves by hand, rational, and does not have
a \(2\)-block rational decomposition, which makes it interesting for the effect on the
edge multiplicities at the end. See \Fig{K6_3} for a standard knot diagram that is
minimal in number of crossings, has \(6\) bridges, and is in a \(3\)-page bridge
embedding. Underpasses are omitted in the figure. From the circle word, these underpasses
are \(\ottt{ef}, \ldots, \ottt{ba}\).

\nextLabel{fig:K6_3}
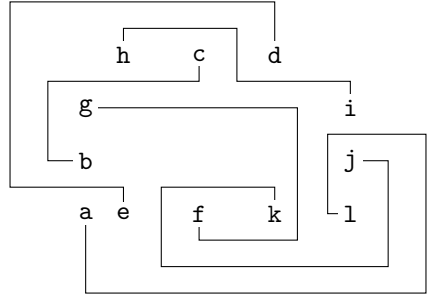
\begin{figure}
  \centering
  \begin{tikzpicture}[x=0.1cm, y=0.07cm, every node/.append style={fill=diagrambg, inner sep=2pt}]
    \coordinate (e) at (10, 0);
    \coordinate (f) at (20, 0);
    \coordinate (k) at (30, 0);
    \coordinate (l) at (40, 0);
    \coordinate (j) at (40, 10);
    \coordinate (i) at (40, 20);
    \coordinate (d) at (30, 30);
    \coordinate (c) at (20, 30);
    \coordinate (h) at (10, 30);
    \coordinate (g) at (5, 20);
    \coordinate (b) at (5, 10);
    \coordinate (a) at (5, 0);
    \draw (b) -- ($(b) + (-5, 0)$) -- ($(g) + (-5, 5)$) -- ($(c) + (0, -5)$) -- (c);
    \draw (d) -- ($(d) + (0, 10)$) -- (-5, 40) -- ($(b) + (-10, -5)$) -- (10, 5) -- (e);
    \draw (f) -- ($(f) + (0, -5)$) -- ($(k) + (3, -5)$) -- (33, 20) -- (g);
    \draw (h) -- ($(h) + (0, 5)$) -- ($(c) + (5, 5)$) -- ($(c) + (5, -5)$) -- ($(i) + (0, 5)$) -- (i);
    \draw (j) -- ($(j) + (5, 0)$) -- (45, -10) -- ($(e) + (5, -10)$) -- ($(e) + (5, 5)$) -- ($(k) + (0, 5)$) -- (k);
    \draw (l) -- ($(l) + (-3, 0)$) -- ($(j) + (-3, 5)$) -- ($(j) + (10, 5)$) -- (50, -15) -- (5, -15) -- (a);
    \node at (a) {\ttt{a}};
    \node at (b) {\ttt{b}};
    \node at (c) {\ttt{c}};
    \node at (d) {\ttt{d}};
    \node at (e) {\ttt{e}};
    \node at (f) {\ttt{f}};
    \node at (g) {\ttt{g}};
    \node at (h) {\ttt{h}};
    \node at (i) {\ttt{i}};
    \node at (j) {\ttt{j}};
    \node at (k) {\ttt{k}};
    \node at (l) {\ttt{l}};
  \end{tikzpicture}
  \caption{Knot \(6_3\) with sentence \ttt{-bhc-dbe-fkg-hdi-jek+lja} \ttt{;efkljidchgba}
    encoding a \(3\)-page bridge presentation.}
  \label{\pendingLabel}
\end{figure}

We perform reduction by the grammar rules. This is shown in \Table{K6_3}. Some normalizations
have been omitted for brevity. In the table, we write the underpass avoidances in four
forms with \(\ttt{X}'_{P}\) giving the word fragment to use when traversing bridge
\ttt{X} to avoid vertex \(\ttt{x}^{\circ}\) in an approach through page \(P\).
Because each bridge has two names and \(P\) can be \(N\) or \(S\), this gives us \(4\)
distinct fragments for each bridge, according to the end being avoided and the page
with the first arc in the avoidance.
These all arise from the same loop around the same bridge. Because using an underpass
avoidance move is followed by closing the underpass, we never use both routes for underpass
avoidance, but we list two versions in this table for clarity. Similarly, we never need
both names for a bridge, but we include both in the table because it makes the selections
clearer during underpass avoidance.

\Table{K6_3} contains enough information to reconstruct the original sentence (hence,
the diagram). Because the underpass avoidance word \(\ttt{X}'_S\) is just the word
\(\ttt{X}'_N\) in reversed order, we do not write both versions in \Table{K6_3reduce}.

\nextLabel{tab:K6_3}
\begin{table}
  \caption{Potential initial grammar rules of knot \(6_3\)}\label{\pendingLabel}
  \begin{tabular}{rcl|rcl|rcl|rcl}
  \(\ttt{B},\ttt{C}\) &\(=\) &\(\ttt{-}\ottt{b}\ttt{h}\ottt{c}\)
    & \(\ttt{D}, \ttt{E}\) &\(=\) &\(\ttt{-}\ottt{d}\ttt{b}\ottt{e}\)
    & \(\ttt{F}, \ttt{G}\) &\(=\) &\(\ttt{-}\ottt{f}\ttt{k}\ottt{g}\)
    & \(\ttt{H}, \ttt{I}\) &\(=\) &\(\ttt{-}\ottt{h}\ttt{d}\ottt{i}\)\\
  \(\ttt{J}, \ttt{K}\) &\(=\) &\(\ttt{-}\ottt{j}\ttt{e}\ottt{k}\)
    & \(\ttt{L}, \ttt{A}\) &\(=\) &\(\ttt{+}\ottt{l}\ttt{j}\ottt{a}\)
    & \ttt{b} & \(=\) & \(\ttt{a}^{\circ},\ttt{b}^{\circ}\)
    & \ttt{d} & \(=\) & \(\ttt{c}^{\circ},\ttt{d}^{\circ}\)\\
  \ttt{e} & \(=\) & \(\ttt{e}^{\circ},\ttt{f}^{\circ}\)
    & \ttt{h} & \(=\) & \(\ttt{g}^{\circ},\ttt{h}^{\circ}\)
    & \ttt{j} & \(=\) & \(\ttt{i}^{\circ},\ttt{j}^{\circ}\)
    & \ttt{k} & \(=\) & \(\ttt{k}^{\circ},\ttt{l}^{\circ}\)\\
  \(\ttt{A}'_N\) &\(=\) & \(\ttt{ajlkj}\)
    & \(\ttt{A}'_S\) &\(=\) & \(\ttt{jklja}\)
    & \(\ttt{B}'_N\) &\(=\) & \(\ttt{ghdch}\)
    & \(\ttt{B}'_S\) &\(=\) & \(\ttt{hcdhg}\)\\
  \(\ttt{C}'_N\) &\(=\) & \(\ttt{hgbhc}\)
    & \(\ttt{C}'_S\) &\(=\) & \(\ttt{chbgh}\)
    & \(\ttt{D}'_N\) &\(=\) & \(\ttt{ibaeb}\)
    & \(\ttt{D}'_S\) &\(=\) & \(\ttt{beabi}\)\\
  \(\ttt{E}'_N\) &\(=\) & \(\ttt{bdiba}\)
    & \(\ttt{E}'_S\) &\(=\) & \(\ttt{abidb}\)
    & \(\ttt{F}'_N\) &\(=\) & \(\ttt{fkghk}\)
    & \(\ttt{F}'_S\) &\(=\) & \(\ttt{khgkf}\)\\
  \(\ttt{G}'_N\) &\(=\) & \(\ttt{kefkg}\)
    & \(\ttt{G}'_S\) &\(=\) & \(\ttt{gkfek}\)
    & \(\ttt{H}'_N\) &\(=\) & \(\ttt{cdjid}\)
    & \(\ttt{H}'_S\) &\(=\) & \(\ttt{dijdc}\)\\
  \(\ttt{I}'_N\) &\(=\) & \(\ttt{dchdi}\)
    & \(\ttt{I}'_S\) &\(=\) & \(\ttt{idhcd}\)
    & \(\ttt{J}'_N\) &\(=\) & \(\ttt{lefke}\)
    & \(\ttt{J}'_S\) &\(=\) & \(\ttt{ekfel}\)\\
  \(\ttt{K}'_N\) &\(=\) & \(\ttt{ejlef}\)
    & \(\ttt{K}'_S\) &\(=\) & \(\ttt{felje}\)
    & \(\ttt{L}'_N\) &\(=\) & \(\ttt{jbajl}\)
    & \(\ttt{L}'_S\) &\(=\) & \(\ttt{ljabj}\)\\
  \end{tabular}
\end{table}

\Table{K6_3reduce} ends with two bridges, \ttt{D} and \ttt{H}. 
Because \(\ttt{d} = \ttt{d}^{\circ} \in \ttt{D}\) and \(\ttt{h} = \ttt{h}^{\circ} \in \ttt{H}\),
the bridge presentation is locally minimal. See \Fig{K6_3circle} for a chord diagram of the
bridge presentation represented by the sentence \ttt{-hdghdchgdhcdhgdchdghc}
\ttt{-dhcdhgdchdghdchgdhcdg;dchg}. Starting at \Fig{K6_3} and ending at \Fig{K6_3circle}, we
show the evolution of the diagram as different underpasses are closed and the number of
bridges goes down.

\nextLabel{tab:K6_3reduce}
\begin{table}
  \caption{Reducing the number of bridges in knot \(6_3\) to \(2\).}\label{\pendingLabel}
  \begin{tabular}{rcl|rcl}
    \multicolumn{3}{c}{\textbf{Rules}} & \multicolumn{3}{c}{\textbf{Changes}}\\
    \hline
    \ttt{b} & \(\mapsto\) & \(\ttt{B}'\)
      & \ttt{D} & \(=\) & \(\ttt{-}\ottt{d}\ttt{B}_S'\ottt{e}\)\\
    & & & & \(=\) & \ttt{-}\ottt{d}\ttt{hcdhg}\ottt{e}\\
    & & & & \(=\) & \ttt{-}\ottt{d}\ttt{hcdh}\ottt{e}\\
    \ttt{L} & \(\mapsto\) & \ttt{L\#C}
      & \ttt{L} & \(=\) & \ttt{+}\ottt{l}\ttt{jh}\ottt{c}\\
    \ttt{B} & \(\mapsto\) & \(\emptyset\) & & &\\
    \ottt{a},\ottt{b} & \(\mapsto\) & \(\emptyset\) & & &\\
    \ttt{a},\ttt{b} & \(\mapsto\) & \ttt{g} & \multicolumn{3}{c}{(See \Fig{K6_3_2})}\\
    \ttt{e} & \(\mapsto\) & \(\ttt{F}'\) & \ttt{J} & \(=\) & \ttt{-}\ottt{j}\(\ttt{F}'_S\)\ottt{k}\\
    & & & & \(=\) & \ttt{-}\ottt{j}\ttt{khghkf}\ottt{k}\\
    & & & & \(=\) & \ttt{-}\ottt{j}\ttt{khgk}\ottt{k}\\
    & & & & \(=\) & \ttt{-}\ottt{j}\ttt{khg}\ottt{k}\\
    \ttt{D} & \(\mapsto\) & \ttt{D\#G} & \ttt{D} & \(=\) & \ttt{-}\ottt{d}\ttt{hcdhk}\ottt{g}\\
    \ttt{F} & \(\mapsto\) & \(\emptyset\) & & &\\
    \ottt{e},\ottt{f} & \(\mapsto\) & \(\emptyset\) & & &\\
    \ttt{e},\ttt{f} & \(\mapsto\) & \ttt{g} & \multicolumn{3}{c}{(See \Fig{K6_3_3})}\\
    \ttt{j} & \(\mapsto\) & \(\ttt{I}'\) & \ttt{L} & \(=\) & \ttt{+}\ottt{l}\(\ttt{I}'_N\)\ttt{h}\ottt{c}\\
    & & & & \(=\) & \ttt{+}\ottt{l}\ttt{dchdlh}\ottt{c}\\
    \ttt{H} & \(\mapsto\) & \ttt{H\#K} & \ttt{H} & \(=\) & \ttt{-}\ottt{h}\ttt{dlkh}\ottt{k}\\
    \ttt{I} & \(\mapsto\) & \(\emptyset\) & & &\\
    \ottt{i},\ottt{j} & \(\mapsto\) & \(\emptyset\) & & &\\
    \ttt{i},\ttt{j} & \(\mapsto\) & \ttt{l} & \multicolumn{3}{c}{(See \Fig{K6_3_4})}\\
    \ttt{k} & \(\mapsto\) & \(\ttt{L}'\) & \ttt{D} & \(=\) & \ttt{-}\ottt{d}\(\ttt{hcdhL}'_S\)\ottt{g}\\
    & & & & \(=\) & \ttt{-}\ottt{d}\ttt{hcdhldchdlhdchldhcd}\ottt{g}\\
    & & & \ttt{H} & \(=\) & \ttt{-}\ottt{h}\(\ttt{dlL}'_S\)\ttt{h}\ottt{k}\\
    & & & & \(=\) & \ttt{-}\ottt{h}\ttt{dlldchdlhdchldhcdh}\ottt{k}\\
    & & & & \(=\) & \ttt{-}\ottt{h}\ttt{ddchdlhdchldhcdh}\ottt{k}\\
    & & & & \(=\) & \ttt{-}\ottt{h}\ttt{chdlhdchldhcdh}\ottt{k}\\
    & & & & \(=\) & \ttt{-}\ottt{h}\ttt{dlhdchldhcdh}\ottt{k}\\
    \ttt{H} & \(\mapsto\) & \ttt{H\#C} & \ttt{H} & \(=\) & \ttt{-}\ottt{h}\ttt{dlhdchldhcdhkdchdlh}\ottt{c}\\
    \ttt{L} & \(\mapsto\) & \(\emptyset\) & & &\\
    \ottt{k},\ottt{l} & \(\mapsto\) & \(\emptyset\) & & &\\
    \ttt{k},\ttt{l} & \(\mapsto\) & \ttt{g} & & &\\
  \end{tabular}
\end{table}

\nextLabel{fig:K6_3_2}
\begin{figure}
  \centering
  \begin{tikzpicture}[x=0.1cm, y=0.07cm, every node/.append style={fill=diagrambg, inner sep=2pt}]
    \coordinate (e) at (10, 0);
    \coordinate (f) at (20, 0);
    \coordinate (k) at (30, 0);
    \coordinate (l) at (40, 0);
    \coordinate (j) at (40, 10);
    \coordinate (i) at (40, 20);
    \coordinate (d) at (30, 25);
    \coordinate (c) at (20, 25);
    \coordinate (h) at (10, 25);
    \coordinate (g) at (10, 10);
    \coordinate (b) at (5, 10);
    \coordinate (a) at (5, 0);
    \draw (d) -- ($(d) + (0, 10)$) -- ($(h) + (-5, 10)$) -- ($(h) + (-5, -3)$) -- ($(h) + (5, -3)$) -- ($(c) + (-5, 2.5)$) -- ($(c) + (2.5, 2.5)$) -- ($(c) + (2.5, -10)$) -- ($(g) + (-5, 5)$) -- ($(e) + (-5, 0)$) -- (e);
    \draw (f) -- ($(f) + (0, -5)$) -- ($(k) + (3, -5)$) -- (33, 10) -- (g);
    \draw (h) -- ($(h) + (0, 5)$) -- ($(c) + (5, 5)$) -- ($(c) + (5, -5)$) -- (i);
    \draw (j) -- ($(j) + (5, 0)$) -- (45, -10) -- ($(e) + (5, -10)$) -- ($(e) + (5, 5)$) -- ($(k) + (0, 5)$) -- (k);
    \draw (l) -- ($(l) + (-3, 0)$) -- ($(j) + (-3, 5)$) -- ($(j) + (10, 5)$) -- (50, -15) -- (0, -15) -- ($(g) + (-10, 8)$) -- ($(c) + (0, -7)$) -- (c);
    \node at (c) {\ttt{c}};
    \node at (d) {\ttt{d}};
    \node at (e) {\ttt{e}};
    \node at (f) {\ttt{f}};
    \node at (g) {\ttt{g}};
    \node at (h) {\ttt{h}};
    \node at (i) {\ttt{i}};
    \node at (j) {\ttt{j}};
    \node at (k) {\ttt{k}};
    \node at (l) {\ttt{l}};
  \end{tikzpicture}
  \caption{Knot \(6_3\) with sentence \ttt{-dhcdhe-fkg-hdi-jek+ljhc} \ttt{;efkljidchg}
    encoding a \(3\)-page bridge presentation.}
  \label{\pendingLabel}
\end{figure}

\nextLabel{fig:K6_3_3}
\begin{figure}
  \centering
  \begin{tikzpicture}[x=0.1cm, y=0.07cm, every node/.append style={fill=diagrambg, inner sep=2pt}]
    \coordinate (k) at (25, 0);
    \coordinate (l) at (35, 0);
    \coordinate (j) at (35, 10);
    \coordinate (i) at (35, 20);
    \coordinate (d) at (30, 25);
    \coordinate (c) at (20, 25);
    \coordinate (h) at (10, 25);
    \coordinate (g) at (10, 5);
    \draw (d) -- ($(d) + (0, 10)$) -- ($(h) + (-5, 10)$) -- ($(h) + (-5, -3)$) -- ($(h) + (5, -3)$) -- ($(c) + (-5, 2.5)$) -- ($(c) + (2.5, 2.5)$) -- ($(c) + (2.5, -10)$) -- ($(g) + (-8, 10)$) -- ($(g) + (-8, -10)$) -- ($(k) + (3, -5)$) -- (28, 5) -- (g);
    \draw (h) -- ($(h) + (0, 5)$) -- ($(c) + (5, 5)$) -- ($(c) + (5, -5)$) -- (i);
    \draw (j) -- ($(j) + (5, 0)$) -- (40, -5) -- ($(l) + (-5, -5)$) -- ($(j) + (-5, 0)$) -- ($(g) + (-5, 5)$) -- ($(g) + (-5, -5)$) -- (k);
    \draw (l) -- ($(l) + (-3, 0)$) -- ($(j) + (-3, 5)$) -- ($(j) + (10, 5)$) -- (45, -10) -- (0, -10) -- ($(g) + (-10, 13)$) -- ($(c) + (0, -7)$) -- (c);
    \node at (c) {\ttt{c}};
    \node at (d) {\ttt{d}};
    \node at (g) {\ttt{g}};
    \node at (h) {\ttt{h}};
    \node at (i) {\ttt{i}};
    \node at (j) {\ttt{j}};
    \node at (k) {\ttt{k}};
    \node at (l) {\ttt{l}};
  \end{tikzpicture}
  \caption{Knot \(6_3\) with sentence \ttt{-dhcdhkg-hdi-jkhk+ljhc} \ttt{;kljidchg}
    encoding a \(3\)-page bridge presentation.}
  \label{\pendingLabel}
\end{figure}

\nextLabel{fig:K6_3_4}
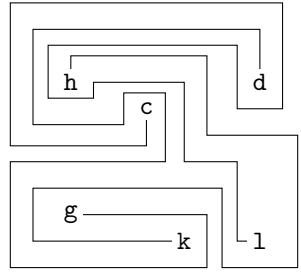
\begin{figure}
  \centering
  \begin{tikzpicture}[x=0.1cm, y=0.07cm, every node/.append style={fill=diagrambg, inner sep=2pt}]
    \coordinate (k) at (25, 0);
    \coordinate (l) at (35, 0);
    \coordinate (j) at (35, 10);
    \coordinate (i) at (35, 20);
    \coordinate (d) at (35, 30);
    \coordinate (c) at (20, 25);
    \coordinate (h) at (10, 30);
    \coordinate (g) at (10, 5);
    \draw (d) -- ($(d) + (0, 10)$) -- ($(h) + (-5, 10)$) -- ($(h) + (-5, -8)$) -- ($(h) + (7, -8)$) -- ($(c) + (-3, 3)$) -- ($(c) + (2.5, 3)$) -- ($(c) + (2.5, -10)$) -- ($(g) + (-8, 10)$) -- ($(g) + (-8, -10)$) -- ($(k) + (3, -5)$) -- (28, 5) -- (g);
    \draw (h) -- ($(h) + (0, 5)$) -- ($(c) + (8, 10)$) -- ($(c) + (8, -5)$) -- ($(i) + (5, 0)$) -- (40, -5) -- ($(l) + (-5, -5)$) -- ($(j) + (-5, 0)$) -- ($(g) + (-5, 5)$) -- ($(g) + (-5, -5)$) -- (k);
    \draw (l) -- ($(l) + (-3, 0)$) -- ($(j) + (-3, 5)$) -- ($(c) + (5, -10)$) -- ($(c) + (5, 5)$) -- ($(h) + (3, 0)$) -- ($(h) + (3, -3)$) -- ($(h) + (-3, -3)$) -- ($(h) + (-3, 7)$) -- ($(d) + (-3, 7)$) -- ($(d) + (-3, -5)$) -- ($(d) + (3, -5)$) -- ($(d) + (3, 15)$) -- ($(h) + (-8, 15)$) -- ($(h) + (-8, -12)$) -- ($(c) + (0, -7)$) -- (c);
    \node at (c) {\ttt{c}};
    \node at (d) {\ttt{d}};
    \node at (g) {\ttt{g}};
    \node at (h) {\ttt{h}};
    \node at (k) {\ttt{k}};
    \node at (l) {\ttt{l}};
  \end{tikzpicture}
  \caption{Knot \(6_3\) with sentence \ttt{-dhcdhkg-hdlkhk-ldchdlhc} \ttt{;kldchg}
    encoding a \(3\)-page bridge presentation.}
  \label{\pendingLabel}
\end{figure}

\nextLabel{fig:K6_3circle}
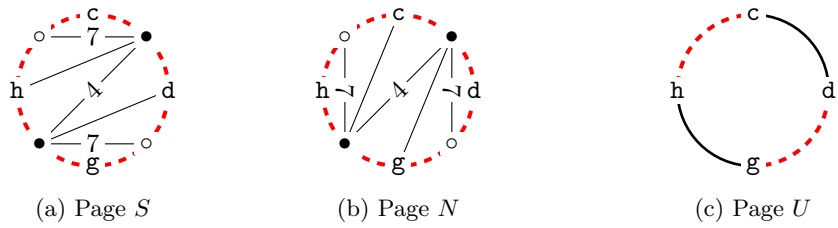
\begin{figure}
  \centering
  \begin{subfigure}[t]{0.31\textwidth}
    \centering
    \begin{tikzpicture}[every node/.append style={fill=diagrambg, inner sep=2pt}]
      \coordinate (v0) at (0:1);
      \coordinate (v0e) at (45:1);
      \coordinate (v1) at (90:1);
      \coordinate (v1e) at (135:1);
      \coordinate (v4) at (180:1);
      \coordinate (v4e) at (225:1);
      \coordinate (v5) at (270:1);
      \coordinate (v5e) at (315:1);
      \draw[local binding] (0, 0) circle[radius=1];
      \draw (v4) -- (v0e);
      \draw (v0) -- (v4e);
      \draw (v5e) -- node[sloped] {\(7\)} (v4e);
      \draw (v1e) -- node[sloped] {\(7\)} (v0e);
      \draw (v0e) -- node[sloped] {\(4\)} (v4e);
      \node at (v0) {\ttt{d}};
      \node at (v0e) {\(\bullet\)};
      \node at (v1) {\ttt{c}};
      \node at (v1e) {\(\circ\)};
      \node at (v4) {\ttt{h}};
      \node at (v4e) {\(\bullet\)};
      \node at (v5) {\ttt{g}};
      \node at (v5e) {\(\circ\)};
    \end{tikzpicture}
    \caption{Page \(S\)}
  \end{subfigure}
  \begin{subfigure}[t]{0.31\textwidth}
    \centering
    \begin{tikzpicture}[every node/.append style={fill=diagrambg, inner sep=2pt}]
      \coordinate (v0) at (0:1);
      \coordinate (v0e) at (45:1);
      \coordinate (v1) at (90:1);
      \coordinate (v1e) at (135:1);
      \coordinate (v4) at (180:1);
      \coordinate (v4e) at (225:1);
      \coordinate (v5) at (270:1);
      \coordinate (v5e) at (315:1);
      \draw[local binding] (0, 0) circle[radius=1];
      \draw (v1e) -- node[sloped] {\(7\)} (v4e);
      \draw (v0e) -- node[sloped] {\(7\)} (v5e);
      \draw (v0e) -- node[sloped] {\(4\)} (v4e);
      \draw (v4e) -- (v1);
      \draw (v0e) -- (v5);
      \node at (v0) {\ttt{d}};
      \node at (v0e) {\(\bullet\)};
      \node at (v1) {\ttt{c}};
      \node at (v1e) {\(\circ\)};
      \node at (v4) {\ttt{h}};
      \node at (v4e) {\(\bullet\)};
      \node at (v5) {\ttt{g}};
      \node at (v5e) {\(\circ\)};
    \end{tikzpicture}
    \caption{Page \(N\)}
  \end{subfigure}
  \hfill
  \begin{subfigure}[t]{0.31\textwidth}
    \centering
    \begin{tikzpicture}[every node/.append style={fill=diagrambg, inner sep=2pt}]
      \def\aZero{0} \def\aOne{90} \def\aFour{180} \def\aFive{270}
      \def\rN{1} \def\rOut{1.25} \def\rOutA{1.5} \def\rOutC{2}
      \coordinate (v0) at (\aZero:1);
      \coordinate (v1) at (\aOne:1);
      \coordinate (v4) at (\aFour:1);
      \coordinate (v5) at (\aFive:1);
      \draw[local binding] (0, 0) circle[radius=1];
      \draw[local strand] (v0) arc[start angle=\aZero, end angle=\aOne, radius=1];
      \draw[local strand] (v4) arc[start angle=\aFour, end angle=\aFive, radius=1];
      \node at (v0) {\ttt{d}};
      \node at (v1) {\ttt{c}};
      \node at (v4) {\ttt{h}};
      \node at (v5) {\ttt{g}};
    \end{tikzpicture}
    \caption{Page \(U\)}
  \end{subfigure}
  \caption{Pages of the \(3\)-page embedding of the \(2\)-bridge presentation of
    knot \(6_3\) with sentence \ttt{-hdghdchgdhcdhgdchdghc}
    \ttt{-dhcdhgdchdghdchgdhcdg;dchg}.}
  \label{\pendingLabel}
\end{figure}

\bibliographystyle{amsplain}
\bibliography{references}

\end{document}